\documentclass[journal, twocolumn, pagenumbers]{IEEEtran}

\usepackage[T1]{fontenc}
\usepackage[latin9]{inputenc}
\usepackage{float}
\usepackage{enumitem}
\usepackage{amsmath}
\usepackage{amsthm}
\usepackage{amssymb}
\usepackage{graphicx}
\usepackage{color}
\usepackage{comment}

\makeatletter

\usepackage{fullpage}
\usepackage{relsize}

\newtheoremstyle{myremark} 
    {\topsep}                    
    {\topsep}                    
    {\rm}                        
    {}                           
    {\bf}                        
    {.}                          
    {.5em}                       
    {}  
    
\theoremstyle{myremark}

\usepackage{listings}
\usepackage{amsmath}
\usepackage{amsthm}
\usepackage{tikz}
\usepackage{caption}
\usepackage{array}
\usepackage{mdwmath}
\usepackage{multirow}
\usepackage{mdwtab}
\usepackage{eqparbox}
\usepackage{amsfonts}
\usepackage{tikz}
\usepackage{multirow,bigstrut,threeparttable}
\usepackage{amsthm}
\usepackage{array}
\usepackage{bbm}
\usepackage{subfigure}
\usepackage{epstopdf}
\usepackage{mdwmath}
\usepackage{mdwtab}
\usepackage{eqparbox}
\usepackage{tikz}
\usepackage{latexsym}
\usepackage{cite}
\usepackage{amssymb}
\usepackage{bm}
\usepackage{amssymb}
\usepackage{graphicx}
\usepackage{mathrsfs}
\usepackage{epsfig}
\usepackage{psfrag}
\usepackage{setspace}
\usepackage{hyperref}
\usepackage{algorithm}
\usepackage{algpseudocode}
\usepackage{stfloats}

\newtheorem{remark}{Remark}

\newtheorem{theorem}{Theorem}
\newtheorem{construction}{Construction}

\newtheorem{lemma}{Lemma} 
\newtheorem{corollary}{Corollary}

\newtheorem{condition}{Condition}
\newtheorem{definition}{Definition}
\newcommand{\poly}{{\mathsf{poly}}}
\newcommand{\norm}[1]{\left\lVert#1\right\rVert}

\def \bE {\mathbb{E}}

\def \bR {\mathbb{R}}

\def \EE {\mathbb{E}}
\allowdisplaybreaks[4]

\begin{document}

\title{Bias Correction with Jackknife, Bootstrap, and Taylor Series}

\author{Jiantao~Jiao,~\IEEEmembership{Member,~IEEE},~Yanjun~Han,~\IEEEmembership{Student Member,~IEEE}
\thanks{Manuscript received Month 00, 0000; revised Month 00, 0000; accepted Month 00, 0000. Date of current version Month 00, 0000.
Copyright (c) 2019 IEEE. Personal use of this material is permitted. However, permission to use this material for any other purposes must be obtained from the IEEE by sending a request to pubs-permissions@ieee.org. The research of Jiantao Jiao was partially supported by NSF Grants IIS-1901252, and CCF-1909499.
	}
\thanks{Jiantao Jiao is with the Department of Electrical Engineering and Computer Sciences and the Department of Statistics, University of California, Berkeley, CA, USA. Email: jiantao@eecs.berkeley.edu.} 
\thanks{Yanjun Han is with the Department of Electrical Engineering, Stanford University, CA, USA. Email: \{yjhan\}@stanford.edu.} 
}

\date{\today}

\maketitle

\begin{abstract}
We analyze bias correction methods using jackknife, bootstrap, and Taylor series. We focus on the binomial model, and consider the problem of bias correction for estimating $f(p)$, where $f \in C[0,1]$ is arbitrary. We characterize the supremum norm of the bias of general jackknife and bootstrap estimators for any continuous functions, and demonstrate the in delete-$d$ jackknife, different values of $d$ may lead to drastically different behaviors in jackknife. We show that in the binomial model, iterating the bootstrap bias correction infinitely many times may lead to divergence of bias and variance, and demonstrate that the bias properties of the bootstrap bias corrected estimator after $r-1$ rounds are of the same order as that of the $r$-jackknife estimator if a bounded coefficients condition is satisfied. 
\end{abstract}


\section{Introduction}

One of the classic problems in statistics is to design procedures to reduce the bias of estimators. General bias correction methods such as the bootstrap, the jackknife, and the Taylor series have been widely employed and well studied in the literature. See~\cite{quenouille1956notes,adams1971asymptotic,  Miller1974jackknife, Efron1979bootstrap, efron1982jackknife, , Withers1987, hall1988bootstrap,hall1992bootstrap, Efron--Tibshirani1994introduction, politis1999springer, zoubir2004bootstrap,shao2012jackknife,  cordeiro2014introduction} and the references therein. The jackknife idea is also closely related to the ensemble method in estimation~\cite{kumar2013ensemble,moon2016improving,delattre2017kozachenko}.  

A close inspection of the literature on those general bias correction methods show that they usually rely on certain expansion (differentiability) properties of the expectation of the estimator one would like to correct the bias for, and the analysis is pointwise asymptotics~\cite{Withers1987,hall1992bootstrap, shao2012jackknife}. 

One motivation for this work is that the methods based on series expansions and differentiability assumptions may \emph{not} suffice in the analysis of bootstrap and jackknife even in the simplest statistical models, and the practical implementations of bootstrap and jackknife do not require those differentiability conditions. The Taylor series itself, by definition, is a series expansion method which we include here for comparison with bootstrap and jackknife. 

To illustrate our point, consider one of the simplest statistical models, the binomial model, where $n \cdot \hat{p}_n\sim \mathsf{B}(n, p)$. For any function $f :[0,1] \mapsto \mathbb{R}$, we would like to correct the bias of $f(\hat{p}_n)$ as an estimator of $f(p)$. Let $e_{1, n}(p) = f(p) - \mathbb{E}_p f(\hat{p}_n)$ be the bias term. The expectation of the jackknife bias corrected estimator $\hat{f}_2$ satisfies
\begin{align}\label{eqn.jackknifestandardbias}
\mathbb{E}[\hat{f}_2] & = \mathbb{E}\left[ nf(\hat{p}_n) - (n-1)f(\hat{p}_{n-1}) \right ], 
\end{align}
where $(n-1) \cdot \hat{p}_{n-1} \sim \mathsf{B}(n-1,p)$. The textbook argument of the bias reduction property of the jackknife is the following~\cite{shao2012jackknife}. Suppose that
\begin{align}\label{eqn.jackknifeassumption}
e_{1,n}(p) = \frac{a(p)}{n} + \frac{b(p)}{n^2} + O_p\left( \frac{1}{n^3} \right),
\end{align}
where $a(p),b(p)$ are unknown functions of $p$ which do no depend on $n$, and $O_p(a_n)$ is a sequence that is elementwise upper bounded by $a_n$ up to a multiplicative constant for fixed $p$. We also have
\begin{align}
e_{1,n-1}(p) = \frac{a(p)}{n-1} + \frac{b(p)}{(n-1)^2} + O_p\left( \frac{1}{(n-1)^3} \right).
\end{align}

Hence, the overall bias of $\hat{f}_2$ is:
\begin{align}
f(p) - \mathbb{E}_p \hat{f}_2 & = ne_{1,n}(p) - (n-1)e_{1,n-1}(p) \\
& = \frac{b(p)}{n} - \frac{b(p)}{n-1} + O_p\left( \frac{1}{n^2} \right) \\
& = -\frac{b(p)}{n(n-1)} + O_p\left( \frac{1}{n^2} \right),\label{eqn.introductionf2reducedbias}
\end{align}
which seems to suggest that the bias has been reduced to order $\frac{1}{n^2}$ instead of order $\frac{1}{n}$. However, if we particularize~(\ref{eqn.jackknifeassumption}) to $f(p) = p\ln (1/p)$, which relates to the Shannon entropy~\cite{Shannon1948}, we have~\cite{Harris1975}
\begin{align}\label{eqn.entropyexpansion}
e_{1,n}(p) & = \frac{1-p}{2n} + \frac{1}{12 n^2}\left( \frac{1}{p} - p \right) + O_p \left( \frac{1}{n^3} \right). 
\end{align}
One immediately sees that it may not be reasonable to claim that the jackknife has reduced the bias upon looking at~(\ref{eqn.entropyexpansion}) and~(\ref{eqn.introductionf2reducedbias}). Indeed, the bias of the jackknife estimator is uniformly upper bounded by $O(n)$, but the right hand side of~(\ref{eqn.entropyexpansion}) and~(\ref{eqn.introductionf2reducedbias}) explodes to infinity as $p\to 0$. It shows that one cannot ignore the dependence on $p$ in the $O_p(\cdot)$ notation, but even doing higher order of Taylor expansion does not help. In fact, it was shown first in~\cite{Paninski2003} that for $f(p) = p\ln(1/p)$, there exist universal constants $C_1>0,C_2>0$ such that for any $n\geq 1$,
\begin{align}
\sup_{p\in [0,1]} | e_{1,n}(p) | & \leq \frac{C_1}{n} \\
\sup_{p\in [0,1]} | f(p) - \mathbb{E}_p[\hat{f}_2]| & \geq \frac{C_2}{n}. 
\end{align}
In other words, the jackknife does not change the bias order at all. There exist other estimators that achieve a smaller order of bias. Indeed, the estimators~\cite{Valiant--Valiant2011power, Valiant--Valiant2013estimating,Wu--Yang2014minimax, Jiao--Venkat--Han--Weissman2015minimax} that achieve the optimal minimax sample complexity for Shannon entropy estimation used best approximation polynomials to reduce the bias of of each symbol from $\frac{1}{n}$ to $\frac{1}{n\ln n}$. 

In this paper, we connect the jackknife and bootstrap to the theory of approximation, and provide a systematic treatment of the problem of correcting the bias for $f(\hat{p}_n)$ as an estimator of $f(p)$ for $f\in C[0,1]$ and $n\cdot \hat{p}_n \sim \mathsf{B}(n,p)$. Compared with existing literature, we choose to simplify the statistical model to the extreme, but consider arbitrary functions $f$. We believe it is an angle worth investigating due to the following reasons. First of all, it directly leads to analysis of the bias correction properties of jackknife and bootstrap for important statistical questions such as the Shannon entropy estimation, which existing theory proves insufficient of handling. Second, even in this simplest statistical model the analysis of jackknife and bootstrap is far non-trivial, and there still exist abundant open problems as we discuss in this paper. One message we would like to convey in this work is that the analysis of jackknife and bootstrap in simple statistical models but general functions could lead to interesting and deep mathematical phenomena that remain fertile ground for research. We mention that most of the results in this paper could be generalized to the case of natural exponential family of quadratic variance functions~\cite{morris1982natural, morris1983natural}, which comprises of six families: Gaussian, Poisson, binomial, negative binomial, gamma, and generalized hyperbolic secant. Coincidentally, these distribution families were also identified as special in approximation theory literature, where they were named operators of the exponential-type~\cite{may1976saturation,ismail1978family}. 

We introduce some notations below. The $r$-th symmetric difference of a function $f: [0,1] \mapsto \mathbb{R}$ is given by 
\begin{align} \label{eqn.rsymmetricdifferencedef}
\Delta^r_h f(x) = \sum_{k = 0}^r (-1)^k {r \choose k} f(x + r(h/2) - kh),
\end{align}
where $\Delta^r_h f(x) = 0$ if $x + rh/2$ or $x-rh/2$ is not in $[0,1]$. We introduce the $r$-th Ditzian--Totik modulus of smoothness of a function $f :[0,1] \mapsto \mathbb{R}$ as
\begin{align}\label{eqn.dtmodulusdefinition}
\omega_{\varphi}^r(f,t) = \sup_{0<h\leq t} \norm{ \Delta^r_{h\varphi(x)} f },
\end{align}
where $\varphi(x) = \sqrt{x(1-x)}$, and the norm is the supremum norm. 

The $\omega_\varphi^r(f,t)$ modulus satisfies the following properties. 
\begin{lemma}~\cite[Chap. 4]{Ditzian--Totik1987} \label{lemma.dtmodulusproperties}
The Ditzian--Totik modulus of smoothness $\omega_\varphi^r(f,t)$ in~(\ref{eqn.dtmodulusdefinition}) satisfies the following:
\begin{enumerate}
\item $\omega_\varphi^r(f,t)$ is a nondecreasing function of $t$.
\item There exist universal constants $K>0,t_0>0$ such that $\omega_\varphi^r(f, \lambda t) \leq K \lambda^r \omega_\varphi^r(f,t)$ for $\lambda\geq 1$, and $\lambda t \leq t_0$. 
\item There exist universal constants $K>0,t_0>0$ such that $\omega_\varphi^{r+1}(f,t) \leq K \omega_\varphi^r(f,t)$ for $0<t\leq t_0$. 
\item There exists a universal constant $K>0$ such that $\omega_\varphi^{r}(f,t) \leq K \sup_{x\in [0,1]} |f^{(r)}(x)| t^r$. 
\item $\lim_{t\to 0^+} \frac{\omega_\varphi^r(f,t)}{t^r} = 0 \Rightarrow f$ is a polynomial with degree $r-1$. ($f$ is a polynomial of degree $r-1\Rightarrow \omega_\varphi^r(f,t) = 0$). 
\item $\omega_\varphi^r(f,t) = O(t^r)$ for fixed $f,r$ if and only if $f^{(r-1)} \in \text{A.C.}_{\text{loc}}$ and $\| \varphi^r f^{(r)}\| <\infty$,
\end{enumerate}
where $f^{(r-1)} \in \text{A.C.}_{\text{loc}}$ means that $f$ is $r-1$ times differentiable and $f^{(r-1)}$ is absolutely continuous in every closed finite interval $[c,d] \subset (0,1)$.  
\end{lemma}

We emphasize that Ditzian--Totik modulus of smoothness is easy to compute for various functions. For example, for $f(x) = x^\delta | \ln x/2|^\gamma, x\in (0,1)$. Then for $r\geq 2\delta$, we have~\cite[Section 3.4]{Ditzian--Totik1987}:
\begin{align}\label{eqn.dtmoduluscomputationexample}
\omega_\varphi^r(f,t) \asymp_{r,\delta,\gamma} \begin{cases} t^{2\delta} |\ln t|^\gamma & \delta \notin \mathbb{Z},\delta\geq 0, \gamma\geq 0 \\ t^{2\delta} |\ln t|^{\gamma-1} & \delta \in \mathbb{Z},  \delta\geq 0,\gamma \geq 1\end{cases}. 
\end{align}

An intuitive understanding of $\omega_\varphi^r(f,t)$ is the following. If the function $f$ is ``smoother'', the modulus is smaller. However, a non-zero $\omega_\varphi^r(f,t)$ cannot vanish faster than the order $t^r$ for any fixed $f$. 

\emph{Notation:} All the norms in this paper refer to the supremum norm. Concretely $\| f \| = \sup_x |f(x)|$. For non-negative sequences $a_\gamma,b_\gamma$, we use the notation $a_\gamma \lesssim_{\alpha}  b_\gamma$ to denote that there exists a universal constant $C$ that only depends on $\alpha$ such that $\sup_{\gamma } \frac{a_\gamma}{b_\gamma} \leq C$, and $a_\gamma \gtrsim_\alpha b_\gamma$ is equivalent to $b_\gamma \lesssim_\alpha a_\gamma$. Notation $a_\gamma \asymp_\alpha b_\gamma$ is equivalent to $a_\gamma \lesssim_\alpha  b_\gamma$ and $b_\gamma \lesssim_\alpha a_\gamma$. We write $a_\gamma \lesssim b_\gamma$ if the constant is universal and does not depend on any parameters. Notation $a_\gamma \gg b_\gamma$ means that $\liminf_\gamma \frac{a_\gamma}{b_\gamma} = \infty$, and $a_\gamma \ll b_\gamma$ is equivalent to $b_\gamma \gg a_\gamma$. We write $a\wedge b=\min\{a,b\}$ and $a\vee b=\max\{a,b\}$. Moreover, $\poly_n^d$ denotes the set of all $d$-variate polynomials of degree of each variable no more than $n$, and $E_n[f;I]$ denotes the distance of the function $f$ to the space $\poly_n^d$ in the uniform norm $\|\cdot\|_{\infty,I}$ on $I\subset \bR^d$. The space $\poly_n^1$ is also abbreviated as $\poly_n$. All logarithms are in the natural base. The notation $\mathbb{E}_\theta[X]$ denotes the mathematical expectation of the random variable $X$ whose distribution is indexed by the parameter $\theta$. The $s$-backward difference of a function defined over integers $G_n$ is 
\begin{align} \label{eqn.backwardifferencedefine}
\Delta^s G_n \triangleq \sum_{k = 0}^s (-1)^k \binom{s}{k} G_{n-k}. 
\end{align}

\begin{remark}[Operator view of bias reduction]
It was elaborated in~\cite{Jiao--Venkat--Han--Weissman2014maximum} that for any statistical model, the quantity $\mathbb{E}_\theta F(\hat{\theta})$ could be viewed as an operator that maps the function $F(\cdot)$ to another function of $\theta$. The operator is obviously linear in $F$, and is also positive in the sense that if $F\geq 0$ everywhere it is also everywhere non-negative. If we view $\mathbb{E}_\theta F(\hat{\theta})$ as an approximation of $F(\theta)$, then analyzing the bias of the estimator $F(\hat{\theta})$ is equivalent to analyzing the approximation error of $\mathbb{E}_\theta F(\hat{\theta})$. 

Casting the bias analysis problem as an approximation problem, a key observation of this work is that the bootstrap bias correction could be viewed as the \emph{iterated Boolean sum} approximation, and the jackknife bias correction could be viewed as a \emph{linear combination} approximation, and the Taylor series bias correction corresponds to the Taylor series approximation. The main tool we use to handle these approximation theoretic questions is the $K$-functional, which we introduce in Appendix~\ref{sec.kfunctional}.  
\end{remark}

We now summarize our main results for jackknife, bootstrap, and Taylor series bias correction. 

\subsection{Jackknife bias correction}

The jackknife is a subsampling technique~\cite{shao2012jackknife} that aims at making the biases of estimators with different sample sizes cancel each other. Before we introduce the jackknife estimator, we recall the definition of $U$-statistic. 

\begin{definition}[$U$-statistic]
Let $g: \mathbb{R}^r \mapsto \mathbb{R}$ be a real-valued function of $r$ variables. For each $n\geq r$ the associated $U$-statistic $\mathbb{U}_n[g]: \mathbb{R}^n \mapsto \mathbb{R}$ is defined as
\begin{align}
\mathbb{U}_n[g](X_1,X_2,\ldots,X_n) = \frac{\sum_\beta g(X_{\beta_1}, X_{\beta_2},\ldots, X_{\beta_r}) }{{n \choose r}},
\end{align}
which is the average over ordered samples $(\beta_1,\beta_2,\ldots,\beta_r)$ of size $r$ of the sample values $g(X_{\beta_1}, X_{\beta_2},\ldots, X_{\beta_r})$. 
\end{definition}

\begin{definition}[$r$-jackknife estimator]\label{def.generalrjackknifeestimatordefinition}
Fix $r\geq 1, r\in \mathbb{Z}$. Fix $K>0$ such that $K$ does not scale with $n$. For a given function $f\in C[0,1]$ and any positive integer $m$, define function $g_m: \mathbb{R}^m \mapsto \mathbb{R}$ as
\begin{align}
g_m(x_1,x_2,\ldots,x_m) \triangleq f\left( \frac{\sum_{i = 1}^m x_i}{m} \right).
\end{align}

For a collection of sample sizes $n_1<n_2<n_3<\ldots < n_r \leq K n_1 = n$, under the binomial model the general $r$-jackknife estimator is defined as
\begin{align}\label{eqn.generalrjackknifedef}
\hat{f}_r = \sum_{i = 1}^r C_i \mathbb{U}_n[g_m](X_1,X_2,\ldots,X_n),
\end{align}
where $X_1,X_2,\ldots,X_n \stackrel{\text{i.i.d.}}{\sim} \mathsf{Bern}(p)$. The coefficients $\{C_i\}$ are given by
\begin{align}\label{eqn.ciexpressions}
C_i = \prod_{j\neq i} \frac{n_i}{n_i - n_j}, 1\leq i \leq r.
\end{align}

If $n_r = n, n_i - n_{i-1} = d$, then it is called the delete-$d$ $r$-jackknife estimator. 
\end{definition}

Note that the standard jackknife in~(\ref{eqn.jackknifestandardbias}) corresponds to $n_1 = n-1, n_2 = n$, whose corresponding coefficients are $C_1 = -(n-1), C_2 = n$. As shown in Lemma~\ref{lemma.higherpower} in Appendix~\ref{sec.auxlemmas}, the coefficients $\{C_i\}_{1\leq i\leq r}$ in~(\ref{eqn.ciexpressions}) satisfy the following: 
\begin{align}
\sum_{i = 1}^r C_i  &= 1 \nonumber \\
\sum_{i = 1}^r \frac{C_i}{n_i^{\rho}} &= 0,\quad 1\leq \rho \leq r-1, \rho \in \mathbb{Z}.  \label{eqn.vonderm}
\end{align}

The intuition behind this condition is clear: only through these equations can one completely cancel any bias terms of order $1/n^\rho, \rho \leq r-1$. It is also clear that (\ref{eqn.vonderm}) corresponds to solving a linear system with the Vandermonde matrix, for which the solution given in~(\ref{eqn.ciexpressions}) is the unique solution due to the fact that all the $n_i$'s are distinct. The rationale above also appeared in~\cite{schucany1971bias}, and the corresponding coefficients $C_i$ were given in the form of determinants. Equation~(\ref{eqn.ciexpressions}) shows that in this special case the coefficients admit a simple expression. 

\subsubsection{Jackknife with the bounded coefficients condition}

We introduce the following condition on $\{C_i\}_{1\leq i\leq r}$ which turns out to be crucial for the bias and variance properties of the general $r$-jackknife. 

\begin{condition}[Bounded coefficients condition]\label{condi.boundedcoeffcondition}
We say that the jackknife coefficients $C_i$ in~(\ref{eqn.ciexpressions}) satisfy the \emph{bounded coefficients condition} with parameter $C$ if there exists a constant $C$ that only depends on $r$ such that
\begin{align}
\sum_{i = 1}^r |C_i| \leq C. 
\end{align}
\end{condition}

One motivation for Condition~\ref{condi.boundedcoeffcondition} is the following. Observe that
\begin{align} \label{eqn.expectationjackknifegeneral}
\mathbb{E}_p \hat{f}_r = \sum_{i = 1}^r C_i \mathbb{E}_p [f(\hat{p}_{n_i})]. 
\end{align}

Viewing $\mathbb{E}_p[f(\hat{p}_{n_i})]$ as an operator that maps $f$ to a polynomial, it is an approximation to $f(p)$, which is the \emph{Bernstein} polynomial~\cite[Chapter 10]{devore1993constructive}. It follows from the Bernstein theorem~\cite[Chap. 7]{alon2004probabilistic} that $\lim_{n\to \infty}\| \mathbb{E}_p f(\hat{p}_{n}) - f(p) \|_\infty = 0$ for any continuous function $f$ on $[0,1]$. Then, one can view the $r$-jackknife as a linear combination of operators. In this sense, Condition~\ref{condi.boundedcoeffcondition} assures that the linear combination as a new operator has bounded norm that is independent of $n$. 

The following theorem quantifies the performance of the general $r$-jackknife under the bounded coefficients condition in Condition~\ref{condi.boundedcoeffcondition}. 

\begin{theorem}\label{thm.jackknife}
Suppose $\{C_i\}_{1\leq i\leq r}$ satisfies Condition~\ref{condi.boundedcoeffcondition} with parameter $C$. Suppose $r\geq 1$ is a fixed integer. Let $\hat{f}_r$ denote the general $r$-jackknife in~(\ref{eqn.generalrjackknifedef}). Then, for any $f\in C[0,1]$, the following is true. 
\begin{enumerate}
\item 
\begin{align}
\| f(p) - \mathbb{E}_p \hat{f}_r \| \lesssim_{r,C}  \omega_\varphi^{2r}(f, 1/\sqrt{n}) + n^{-r} \| f \| 
\end{align}
\item Fixing $0<\alpha <2r$, 
\begin{align}
& \| f(p) - \mathbb{E}_p \hat{f}_r \|  \lesssim_{\alpha,r, C} n^{-\alpha/2} \nonumber \\ 
&\quad \Leftrightarrow  \omega_\varphi^{2r}(f, t) \lesssim_{\alpha,r,C} t^\alpha.
\end{align}
\item Suppose there is a constant $D<2^{2r}$ for which
\begin{align}
\omega_\varphi^{2r}(f, 2t) \leq D \omega_\varphi^{2r}(f,t)\quad \text{for }t\leq t_0.
\end{align}
Then,
\begin{align}
\| f(p) - \mathbb{E}_p \hat{f}_r \|  \asymp_{r, C, D} \omega_\varphi^{2r}(f, 1/\sqrt{n}). 
\end{align}
\item For $r = 1$, 
\begin{align}
\| f(p) - \mathbb{E}_p \hat{f}_r \|  \asymp \omega_\varphi^{2r}(f, 1/\sqrt{n}). 
\end{align}
\end{enumerate}
\end{theorem}

The following corollary of Theorem~\ref{thm.jackknife} is immediate given~(\ref{eqn.dtmoduluscomputationexample}). 
\begin{corollary}\label{cor.goodjackknifefailentropy}
Under the conditions in Theorem~\ref{thm.jackknife}, if $f = -p\ln p$, then
\begin{align}
\| f(p) - \mathbb{E}_p \hat{f}_r \| \asymp_{r,C} \frac{1}{n}. 
\end{align}
If $f(p) = p^\alpha, 0<\alpha<1$, then,
\begin{align}
\| f(p) - \mathbb{E}_p \hat{f}_r \| \asymp_{r,C} \frac{1}{n^\alpha}. 
\end{align}
\end{corollary}

Corollary~\ref{cor.goodjackknifefailentropy} implies that the $r$-jackknife estimator for fixed $r$ does not improve the bias of $f(\hat{p}_n)$ for $f(p) = p\ln (1/p)$, which makes it incapable of achieving the minimax rates of Shannon entropy estimation~\cite{Wu--Yang2014minimax, Jiao--Venkat--Han--Weissman2015minimax}. 

\subsubsection{Jackknife without the bounded coefficients condition}

Theorem~\ref{thm.jackknife} has excluded the case that $n_1 = n-1, n_2 = n$. Clearly, to satisfy the assumptions of Theorem~\ref{thm.jackknife}, we need to require $|n_i - n_{i-1}| \gtrsim n_1$, which puts a minimum gap between the different sample sizes we can use. This begs the question: is this condition necessary to Theorem~\ref{thm.jackknife} to hold? If not, what bad consequences it will lead to? 

From a computational perspective, taking $d$ small in the delete-$d$ jackknife may reduce the computational burden. However, as we now show, the usual delete-$1$ jackknife does not satisfy Theorem~\ref{thm.jackknife} in general and exhibits drastically different bias and variance properties. \footnote{It has been observed in the literature~\cite{shao1989general} that in jackknife variance estimation, which is a different area of application of the jackknife methodology, sometimes it is also necessary to take $d$ large to guarantee consistency. }

We now show that the delete-$1$ jackknife may have bias and variance both diverging to infinity in the worst case. 

\begin{theorem}\label{thm.delete1jackknifeworst}
Let $\hat{f}_r$ denote the delete-$1$ $r$-jackknife estimator. There exists a fixed function $f\in C(0,1]$ that satisfies $\| f \| \leq 1$ such that
\begin{align}\label{eqn.s_zero_bad_bias}
\| \mathbb{E}_p \hat{f}_r - f(p) \| \gtrsim n^{r-1}. 
\end{align} 
If we allow the function $f$ to depend on $n$, then one can have $f\in C[0,1]$. \footnote{We emphasize that if we restrict $f\in C[0,1], \| f\| \leq 1$, and do not allow $f$ to depend on $n$, then one cannot achieve the bound (\ref{eqn.s_zero_bad_bias}). Indeed, noting that $\sum_{i  =1}^r C_i = 1$, the error term can be written as 
\begin{align}
\left | \sum_{i = 1}^r C_i \mathbb{E}_p f(\hat{p}_{n+i-r}) - f(p) \right | & = \left | \sum_{i = 1}^r C_i \left( \mathbb{E}_p f(\hat{p}_{n+i-r}) - f(p) \right) \right | \\
& \leq \sum_{i = 1}^r |C_i| \| \mathbb{E}_p f(\hat{p}_{n+i-r}) - f(p) \|,
\end{align}
It follows from the Bernstein theorem~\cite[Chap. 7]{alon2004probabilistic} that $ \lim_{n\to \infty}\| \mathbb{E}_p f(\hat{p}_{n}) - f(p) \| = 0$ for any continuous function $f$ on $[0,1]$. Hence, for any $f\in C[0,1]$ one has
\begin{align}
\| \mathbb{E}_p \hat{f}_r - f(p) \| = o(n^{r-1}),
\end{align} 
since $\max_{1\leq i\leq r}|C_i| \lesssim n^{r-1}$ for the delete-1 $r$-jackknife. }

Meanwhile, for any $n\geq 4$, there exists a function $f\in C[0,1]$ depending on $n$ such that 
\begin{align}
\| \mathsf{Var}_p(\hat{f}_2) \| \geq \frac{n^2}{e}. 
\end{align}
\end{theorem}

Theorem~\ref{thm.delete1jackknifeworst} shows that in the worst case, the delete-$1$ $r$-jackknife may have bad performances compared to that satisfying Condition~\ref{condi.boundedcoeffcondition}. Before we delve into the refined analysis of delete-$1$ $r$-jackknife, we illustrate the connection between various types of $r$-jackknife estimators. It turns out that the jackknife is intimately related to the divided differences of functions. 
\begin{definition}[Divided difference]\label{def.divideddifference}
The divided difference $f[x_1,x_2,\ldots,x_n]$ of a function $f$ over $n$ distinct points $\{x_1,x_2,\ldots,x_n\}$ is defined as
\begin{align}
f[x_1,x_2,\ldots,x_n] = \sum_{i =1}^n \frac{f(x_i)}{\prod_{j\neq i} (x_i - x_j)}. 
\end{align}
\end{definition}

It follows from~(\ref{eqn.expectationjackknifegeneral}), ~(\ref{eqn.ciexpressions}) and Lemma~\ref{lemma.higherpower} in Appendix~\ref{sec.auxlemmas} that the bias of a general $r$-jackknife estimator $\hat{f}_r$ can be written as
\begin{align}
\mathbb{E}_p [\hat{f}_r] - f(p) & = \sum_{i = 1}^r \prod_{j\neq i} \frac{n_i}{n_i - n_j} \left( \mathbb{E}_p [f(\hat{p}_{n_i})] - f(p) \right) \\
& = \sum_{i  =1}^r \frac{n_i^{r-1}\left( \mathbb{E}_p [f(\hat{p}_{n_i})] - f(p) \right) }{\prod_{j\neq i} (n_i - n_j)}. 
\end{align}

Define $G_{n,f,p} = n^{r-1} \left( \mathbb{E}_p [f(\hat{p}_n)] - f(p) \right)$. Then, the bias of $\hat{f}_r$ can be written as the divided difference of function $G_{\cdot,f,p}$:
\begin{align}\label{eqn.dvdjackknifebias}
\mathbb{E}_p [\hat{f}_r] - f(p) & = G_{\cdot,f,p} [n_1,n_2,\ldots,n_r]. 
\end{align}

It follows from the mean value theorem of divided differences defined over integers in Lemma~\ref{lemma.meanvaluetheorem} of Appendix~\ref{sec.auxlemmas} that for every $p,f$, 
\begin{align}
 &| \mathbb{E}_p [\hat{f}_r] - f(p)| \nonumber \\
& \quad \leq \max_{n_1\leq n\leq n_r} |  G_{\cdot,f,p} [n-r+1,n-r+2,\ldots,n-1,n] |,\label{eqn.delete1worst}
\end{align}
and the right hand side of~(\ref{eqn.delete1worst}) is nothing but the maximum of the bias of the delete-$1$ $r$-jackknife with varying sample sizes. 

Equation~(\ref{eqn.delete1worst}) shows that in terms of the bias, the delete-$1$ jackknife might be the ``worst'' among all $r$-jackknife estimators. However, what is the precise performance of the delete-$1$ $r$-jackknife when the function $f$ is ``smooth''? Does the performance improve compared to Theorem~\ref{thm.delete1jackknifeworst}? We answer this question below. 
\begin{condition}[Condition \emph{$D_s$}]\label{condi.conditionds}
A function $f: [0,1]\mapsto \mathbb{R}$ is said to satisfy the condition $D_s, s\geq 0,s\in \mathbb{Z}$ with parameter $L>0$ if the following is true:
\begin{enumerate}
\item $s = 0$: $f$ is Lebesgue integrable on $[0,1]$ and $\sup_{x\in [0,1]}|f(x)|\leq L$. 
\item $s\geq 1$:
\begin{enumerate}
\item $f^{(s-1)}$ is absolutely continuous on $[0,1]$;
\item $\sup_{x\in [0,1]} |f^{(i)}(x)|\leq L, 0\leq i\leq s$. 
\end{enumerate}  
\end{enumerate}
\end{condition}

\begin{remark}
We mention that if a function $f$ satisfies condition $D_s$, it does not necessarily belong to the space $C^s[0,1]$, where $C^s[0,1]$ denotes the space of functions $f$ on $[0,1]$ such that $f^{(s)}$ is continuous. Indeed, the function $f(x) = x^2 \sin(1/x) \mathbbm{1}(x\in (0,1])$ satisfies condition $D_1$ as a function mapping from $[0,1]$ to $\mathbb{R}$, but it does not belong to $C^1[0,1]$. 
\end{remark}

The performance of the delete-$1$ $r$-jackknife in estimating $f(p)$ satisfying condition $D_s$ is summarized in the following theorem. 
\begin{theorem}\label{thm.jackknife_general_rs}
For any $r\ge 1, s\ge 0$ and $f$ satisfying condition $D_s$ with parameter $L$, let $\hat{f}_r$ be the delete-$1$ $r$-jackknife. Then, 
\begin{align}
\|\EE_p \hat{f}_r - f(p)\| \lesssim_{r,s,L}  \begin{cases}
n^{r-s-1} & \text{if }0\le s\le 2r-2;\\
n^{-(r-\frac{1}{2})} & \text{if }s=2r-1;\\
n^{-r} & \text{if }s\ge 2r.
\end{cases}
\end{align}
\end{theorem}

\begin{theorem}\label{thm.jackknife_rs_converse}
For $1\le s\le 2r-3$, there exists some universal constant $c>0$ such that for any $n\in\mathbb{N}$, there exists some function $f\in C^{s}[0,1]$ such that $\|f\|\le 1, \|f'\|\le 1,\cdots,\|f^{(s)}\|\le 1$, and for delete-1 $r$-jackknife $\hat{f}_r$:
\begin{align}
\|\EE_p\hat{f}_r(\hat{p}_n)-f(p)\| \ge cn^{r-1-s}.
\end{align}
\end{theorem}

\begin{theorem}\label{thm.jackknife_rs_converse_missing}
For integer $2r-2\leq s\leq 2r-1$, there exists some function $f\in C^s[0,1]$ such that $\|f\|\le 1, \|f'\|_\infty\le 1,\cdots,\|f^{(s)}\|_\infty\le 1$, and for delete-1 $r$-jackknife $\hat{f}_r$,
\begin{align}
\liminf_{n\to \infty} \frac{\|\EE_p\hat{f}_r(\hat{p}_n)-f(p)\|}{ n^{-s/2}} > 0. 
\end{align}
Moreover, if $s\geq 2r$, then 
\begin{align}
\|\EE_p\hat{f}_r(\hat{p}_n)-f(p)\|_\infty \gtrsim \frac{1}{n^r}. 
\end{align}
\end{theorem}
\begin{proof}
The first part follows from~(\ref{eqn.delete1worst}) and Theorem~\ref{thm.jackknife}. The second part follows from taking $f(p)$ to be a polynomial of order $2r$ with leading coefficient one. 
\end{proof}

Now we compare the performance of the $r$-jackknife estimator $\hat{f}_r$ with and without Condition~\ref{condi.boundedcoeffcondition}. Under Condition~\ref{condi.boundedcoeffcondition}, we know from Theorem \ref{thm.jackknife} that
\begin{align}
|\EE_p \hat{f}_r(\hat{p}_n) - f(p)| \lesssim_{r,s,L,C} n^{-\min\{r,\frac{s}{2}\}} 
\end{align}
for $f$ satisfying condition $D_s$ with parameter $L$ \footnote{It follows from the proof of Theorem~\ref{thm.jackknife} that the first part of Theorem~\ref{thm.jackknife} also applies to functions $f$ satisfying condition $D_s$. }, where the exponent is better than that of Theorem \ref{thm.jackknife_general_rs}. A pictorial illustration is shown in Figure \ref{fig.exponent}. 

\begin{figure}[h]
\centering
\includegraphics[scale=0.55]{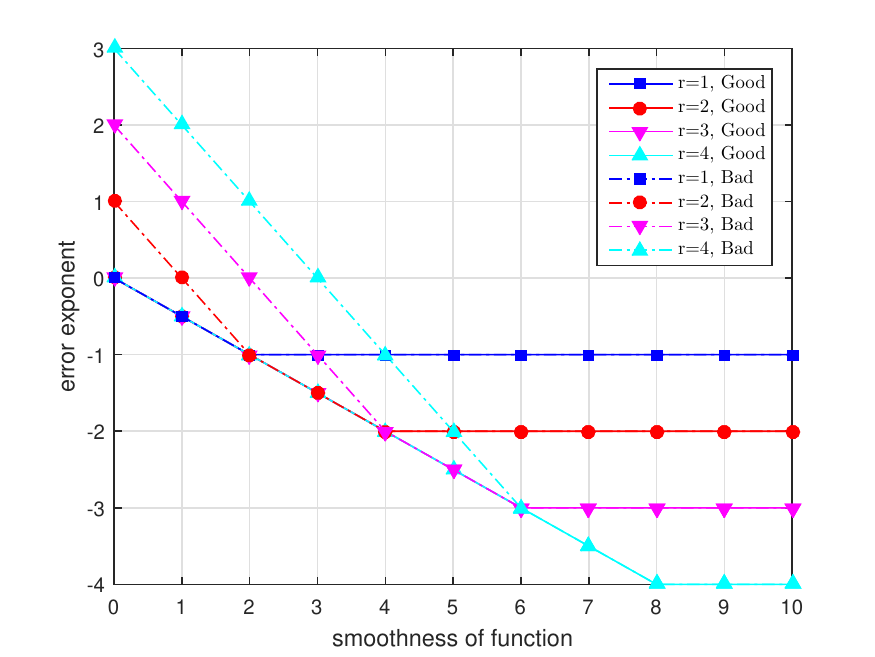}
\caption{Error exponents of ``Good" and ``Bad" jackknife estimators. Here ``Good'' refers to the $r$-jackknife satisfying Condition~\ref{condi.boundedcoeffcondition}, and ``Bad'' refers to the delete-$1$ $r$-jackknife.}\label{fig.exponent}
\end{figure}

\begin{remark}
For general delete-$d$ $r$-jackknife, the cases of $d \asymp n$ and $d = 1$ exhibit drastically different behavior. It remains fertile ground for research to analyze what is the minimum $d$ needed for the delete-$d$ $r$-jackknife to achieve the bias performance that is of the same order as those satisfying Condition~\ref{condi.boundedcoeffcondition} for a specific function $f$. 
\end{remark}

\subsubsection{Specific functions}

The last part of results pertaining to the jackknife investigates some specific functions $f(p)$. Here we take $f(p) = -p\ln p$ or $p^\alpha, 0<\alpha<1$. Those functions even do not belong to $D_1$ under Condition~\ref{condi.conditionds}. However, we show that the jackknife applied to these functions exhibits far better convergence rates than the worst case analysis in Theorem~\ref{thm.jackknife_rs_converse} predicted. 

We show that for the $r$-jackknife when $r = 2$, no matter whether Condition~\ref{condi.boundedcoeffcondition} is satisfied or not, the bias of the jackknife estimator can be universally controlled. 

\begin{theorem}\label{thm.specialfunctionslikeentropy}
Let $\hat{f}_2$ denote a general $2$-jackknife in Definition~\ref{def.generalrjackknifeestimatordefinition}. Then, 
\begin{enumerate}
\item if $f(p) = -p\ln p$, 
\begin{align}
\| \mathbb{E}_p \hat{f}_2 - f(p) \| \lesssim \frac{1}{n}.
\end{align}
\item if $f(p) = p^\alpha, 0<\alpha<1$, 
\begin{align}
\| \mathbb{E}_p \hat{f}_2 - f(p) \| \lesssim_\alpha \frac{1}{n^\alpha}.
\end{align}
\end{enumerate}

Meanwhile, let $\hat{f}_2$ be either the delete-$1$ $2$-jackknife, or a $2$-jackknife that satisfies Condition~\ref{condi.boundedcoeffcondition}. Then, 
\begin{enumerate}
\item if $f(p) = -p\ln p$, 
\begin{align}
\| \mathbb{E}_p \hat{f}_2 - f(p) \| \gtrsim \frac{1}{n}.
\end{align}
\item if $f(p) = p^\alpha, 0<\alpha<1$, 
\begin{align}
\| \mathbb{E}_p \hat{f}_2 - f(p) \| \gtrsim_\alpha \frac{1}{n^\alpha}.
\end{align}
\end{enumerate}
\end{theorem}

\begin{remark}
We conjecture that Theorem~\ref{thm.specialfunctionslikeentropy} holds for any fixed $r$ instead of only $r = 2$. 
\end{remark}

\subsection{Bootstrap bias correction}

The rationale behind bootstrap bias correction is to use the plug-in rule to estimate the bias and then iterate the process~\cite{hall1992bootstrap}. Concretely, suppose we would like to estimate a function $f(\theta)$, and we have an estimator for $\theta$, denoted as $\hat{\theta}$. The estimator $\hat{\theta}(X_1^n)$ is a function of the observations $(X_1,X_2,\ldots,X_n)$, and $X_i \stackrel{\text{i.i.d.}}{\sim} P_\theta$. The bias of the plug-in rule $\hat{f}_1 =  f(\hat{\theta})$ is defined as
\begin{align}
e_1(\theta) = f(\theta) - \mathbb{E}_\theta f(\hat{\theta}). 
\end{align}

We would like to correct this bias. The \emph{additive} bootstrap bias correction does this by using the plug-in rule $e_1(\hat{\theta})$ to estimate $e_1(\theta)$, and then use $f(\hat{\theta}) + e_1(\hat{\theta})$ to estimate $f(\theta)$, hoping that this bias corrected estimator has a smaller bias. 

It is the place that the Monte Carlo approximation principle takes effect: it allows us to compute the plug-in estimator $e_1(\hat{\theta})$ without knowing the concrete form of the bias function $e_1(\theta)$. Indeed, we have
\begin{align}
e_1(\hat{\theta}) = f(\hat{\theta}) - \mathbb{E}_{\hat{\theta}} f(\hat{\theta}^*),
\end{align}
where $\hat{\theta}^*=\hat{\theta}(X_1^*,X_2^*,\ldots,X_n^*)$, and the samples $X_i^* \stackrel{\text{i.i.d.}}{\sim} P_{\hat{\theta}}$. To compute $\mathbb{E}_{\hat{\theta}} f(\hat{\theta}^*)$, it suffices to draw the $n$-tuple sample $(X_1^*,X_2^*,\ldots,X_n^*)$ in total $B$ times under $P_{\hat{\theta}}$, and use the empirical average to replace the expectation, hoping that the law of large number would make the empirical average close to the expectation $\mathbb{E}_{\hat{\theta}} f(\hat{\theta}^*)$. This argument also shows that it takes $B$ rounds of sampling to evaluate $e_1(\cdot)$ at one point. 

After doing bootstrap bias correction as introduced above once, we obtain $\hat{f}_2 =  f(\hat{\theta}) + e_1(\hat{\theta})$. What about its bias? The bias of this new estimator, denoted as $e_2(\theta)$, is 
\begin{align}
e_2(\theta) & = f(\theta) - \left( \mathbb{E}_\theta f(\hat{\theta}) + \mathbb{E}_\theta e_1(\hat{\theta}) \right) \\
& = e_1(\theta) - \mathbb{E}_\theta e_1(\hat{\theta}). 
\end{align}

Clearly, in order to compute $e_2(\hat{\theta})$, we need to evaluate $e_1(\cdot)$ in total $B$ times, which amounts to a total computation complexity $B^2$. 

It motivates the general formula: the bias of the bootstrap bias corrected estimator after $m-1$ rounds of correction is related to that after $m-2$ rounds via
\begin{align}
e_m(\theta) = e_{m-1}(\theta) - \mathbb{E}_\theta e_{m-1}(\hat{\theta}). 
\end{align}
Indeed, denoting the estimator after $m-2$ rounds of bias correction as $\hat{f}_{m-1}$, by definition we know $e_{m-1}(\theta) = f(\theta) - \mathbb{E}_\theta \hat{f}_{m-1}$. The bias corrected estimator after $m-1$ rounds is $\hat{f}_m = \hat{f}_{m-1} + e_{m-1}(\hat{\theta})$, whose bias is
\begin{align}
e_m(\theta) & = f(\theta) - \mathbb{E}_\theta \left( \hat{f}_{m-1} + e_{m-1}(\hat{\theta}) \right) \\
& = e_{m-1}(\theta) - \mathbb{E}_\theta e_{m-1}(\hat{\theta}). 
\end{align}

The bias corrected estimator after $m-1$ rounds of correction is
\begin{align}
\hat{f}_m = f(\hat{\theta}) + \sum_{i = 1}^{m-1} e_i(\hat{\theta}). 
\end{align}

It takes $B^{m-1}$ order computations to compute $\hat{f}_m$ if we view the computation of $\hat{f}_2 = f(\hat{\theta}) + e_1(\hat{\theta})$ takes computational time $B$. We introduce an linear operator $A_n$ that maps the function $f$ to the same function space such that
\begin{align}
A_n[f](\theta) & = \mathbb{E}_\theta f(\hat{\theta}). 
\end{align}
With the help of the operator $A_n$, one may view the bias of $\hat{f}_m$ in the following succinct way. Indeed, since
\begin{align}
e_m(\theta) & = e_{m-1}(\theta) - A_n[e_{m-1}](\theta) \\
& = (I - A_n)[e_{m-1}](\theta), 
\end{align}
we have
\begin{align}
\mathbb{E}_\theta \hat{f}_m & = A_n \left( I + \sum_{i = 1}^{m-1} (I-A_n)^{i} \right)[f] \\
& = A_n\left( \sum_{i = 0}^{m-1} (I-A_n)^i \right)[f] \\
& = (I - (I-A_n)^m)[f].  
\end{align}

The operator $I - (I-A_n)^{m}$ is known as the \emph{iterated Boolean sum} in the approximation theory literature~\cite{natanson1983application,sevy1991acceleration}. Indeed, defining the Boolean sum as $P \oplus Q = P + Q - PQ$, we have
\begin{align}
I - (I-A_n)^{m} = A_n \oplus A_n \oplus \ldots \oplus A_n = \oplus^{m} A_n,
\end{align}
where there are $m$ terms on the right hand side. 

%


Let the bias of the bootstrap bias corrected estimator after $m-1$ rounds be denoted as $e_m(p)$, where 
\begin{align}\label{eqn.bootstraprecurrence}
e_m(p) = e_{m-1}(p) - \mathbb{E}_p e_{m-1}(\hat{p}_n),
\end{align}
and $e_1(p) = f(p) - \mathbb{E}_p[f(\hat{p}_n)]$. Here $f\in C[0,1]$, and $n \cdot\hat{p}_n \sim \mathsf{B}(n,p)$. Our first result on bootstrap bias correction is about the limiting behavior of $e_m(p)$ as $m\to \infty$. In other words, what happens when we conduct the bootstrap bias correction infinitely many times? 

\begin{theorem}\label{thm.bootstraplimit}
Denote the unique polynomial of order $n$ that interpolates the function $f(p)$ at $n+1$ points $\{i/n: 0\leq i \leq n\}$ by $L_n[f]$. Then, for any $f: [0,1] \mapsto \mathbb{R}$, 
\begin{align}
\lim_{m \to \infty} \sup_{p\in [0,1]} | e_m(p) - (f - L_n[f]) | = 0,
\end{align}
where $e_m(p)$ is defined in~(\ref{eqn.bootstraprecurrence}). 
\end{theorem}

In other words, the bias function converges uniformly to the approximation error of the Lagrange interpolation polynomial that interpolates the function $f$ at equidistant points on $[0,1]$. This interpolating polynomial is in general known to exhibit bad approximation properties unless the function is very smooth. The Bernstein example below shows an extreme case. 
\begin{lemma} \cite[Chap. 2, Sec. 2]{natanson1964constructivevol3} [Bernstein's example] \label{lemma.bernsteinbadexample}
Suppose $f(p) = |p-1/2|$, and $L_n[f]$ denotes the unique polynomial that interpolates the function $f(p)$ at $n+1$ points $\{i/n: 0\leq i\leq n\}$. Then, 
\begin{align}
\liminf_{n\to \infty} |L_n[f](p)| = \infty
\end{align}
for all $p\in [0,1]$ except for $p = 0, 1/2, $ and $1$. \footnote{This phenomenon has been generalized to other functions such as $|p-\frac{1}{2}|^\alpha, \alpha>0$ when $\alpha$ is not an even integer. See~\cite{ganzburg2003strong} for more details. }
\end{lemma}
%
For more discussions on the convergence/divergence behavior of $L_n[f]$, we refer the readers to~\cite{mastroianni2008interpolation} for more details. We emphasize that it is a highly challenging question. For example, it was shown in~\cite{li1994local} that for any $p\in [0,1]$, we have
\begin{align}
\lim_{n\to \infty} |L_n[f](p) - f(p)| = 0,
\end{align}
where $f(p) = -p\ln p$ or $p^\alpha, \alpha>0, \alpha \notin \mathbb{Z}$, and $L_n[f]$ is the Lagrange interpolation polynomial at equi-distant points. However, to our knowledge it is unknown that whether $\sup_{p\in [0,1]}|L_n[f](p) - f(p)|$ converges to zero as $n\to \infty$ for those specific functions, and if so, what the convergence rate is.  

As Theorem~\ref{thm.bootstraplimit} and Lemma~\ref{lemma.bernsteinbadexample} show, it may not be a wise idea to iterate the bootstrap bias correction too many times. It is both computationally prohibitive, and even may deteriorate statistically along the process. In practice, one usually conducts the bootstrap bias correction a few times. The next theorem provides performance guarantees for the first few iterations of bootstrap bias correction. 

\begin{theorem}\label{thm.bootstrapafew}
Fix the number of iterations $m\geq 0$, and $0<\alpha \leq 2m$. Then the following statements are true for any $f\in C[0,1]$. Here $e_m(p)$ is defined in~(\ref{eqn.bootstraprecurrence}). 
\begin{enumerate}
\item 
\begin{align}
\|e_m(p)\| \lesssim_m  \omega_\varphi^{2m}(f, 1/\sqrt{n}) + \| f \| n^{-m}. 
\end{align}
\item 
\begin{align}
\|e_m(p)\| \lesssim_{\alpha,m} n^{-\alpha/2} & \Leftrightarrow  \omega_\varphi^{2m}(f,t) \lesssim_{\alpha,m} t^\alpha. 
\end{align}
\item 
\begin{align}
\|e_m(p)\|  \ll_m n^{-m} & \Leftrightarrow f\text{ is an affine function}
\end{align}
\item Suppose there is a constant $D < 2^{2m}$ for which 
\begin{align}
\omega_\varphi^{2m} (f,2t) \leq D \omega_\varphi^{2m} (f,t)\quad\text{for }t\leq t_0.
\end{align} 
Then, 
\begin{align}
\| e_m(p) \| \asymp_{m,D} \omega_\varphi^{2m}(f, 1/\sqrt{n}). 
\end{align}
\item For $m = 1$, 
\begin{align}
\| e_m(p) \| \asymp \omega_\varphi^{2m}(f,1/\sqrt{n}). 
\end{align}
\end{enumerate}
\end{theorem}

Theorem~\ref{thm.bootstrapafew} has several interesting implications. First of all, it shows that for a few iterations of the bootstrap bias correction, we have a decent bound on the bias $\| e_m(p) \|$, which is intimately connected with the $2m$-th order Ditzian--Totik modulus of smoothness evaluated at $1/\sqrt{n}$. This bound is tight in various senses. The second statement shows that it captures the bias $\| e_m(p) \|$ at least up to the granularity of the exponent in $n$, and the third statement shows that it is impossible for the bootstrap bias corrected estimator to achieve bias of order lower than $n^{-m}$ except for the trivial case of affine functions, which have bias zero. The fourth statement shows that as long as the modulus $\omega_\varphi^{2m}(f,t)$ is not too close to $t^{2m}$, the DT modulus bound is tight. The fifth statement shows that when we do not do any bias correction, the DT modulus bound is tight for any function in $C[0,1]$. 

The following corollary is immediate given~(\ref{eqn.dtmoduluscomputationexample}). 
\begin{corollary}
If $f(p) = -p\ln p$, then, 
\begin{align}
\| e_m(p) \| \asymp_m \frac{1}{n} \asymp_m \| e_1(p)\|,
\end{align}
If $f(p) = p^\alpha, 0<\alpha<1$, then
\begin{align}
\| e_m(p) \| \asymp_{m,\alpha} \frac{1}{n^\alpha} \asymp_{m,\alpha} \| e_1(p)\|,
\end{align}
which means that the bootstrap bias correction for the first few rounds does not change the order of bias at all. 
\end{corollary}

We have shown that the bias of $\hat{f}_m$ converges to the approximation error of the Lagrange interpolation polynomial at equi-distant points when $m\to \infty$ (Theorem~\ref{thm.bootstraplimit}). However, we also know that for the first few iterations of the bootstrap, the bias of $\hat{f}_m$ can be well controlled~(Theorem~\ref{thm.bootstrapafew}). It begs the question: how does $\| e_m(p)\|$ evolve as $m\to \infty$?

We study this problem through the example of $f(p) =  |p-1/2|$, with the sample size $n = 20$. Thanks to the special structure of the Binomial functions, we are able to numerically compute $\| e_m(p)\|$ up to $m \approx 8.5\times 10^5$. It follows from Theorem~\ref{thm.bootstraplimit} that 
\begin{align}
\lim_{m\to \infty} \| e_m(p)\| = \| f- L_n[f]\| ,
\end{align}
and for $n = 20$, we numerically evaluated $\| f- L_n[f]\|$ to be $ 47.5945$.

\begin{center}
  \centering
  \centerline{\includegraphics[scale = 0.5]{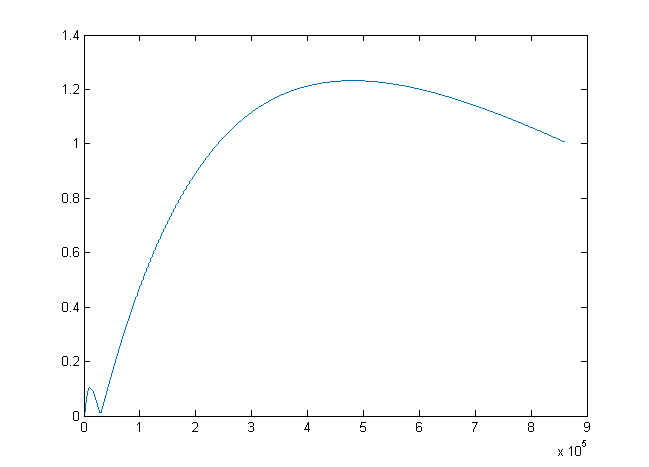}}
\captionof{figure}{The evolution of $\| e_m(p)\|$ as a function of $m$ for $1\leq m\leq 8.5\times 10^5$. }
\label{fig.bootstrapiter}
\end{center}

\begin{center}
  \centering
  \centerline{\includegraphics[scale = 0.5]{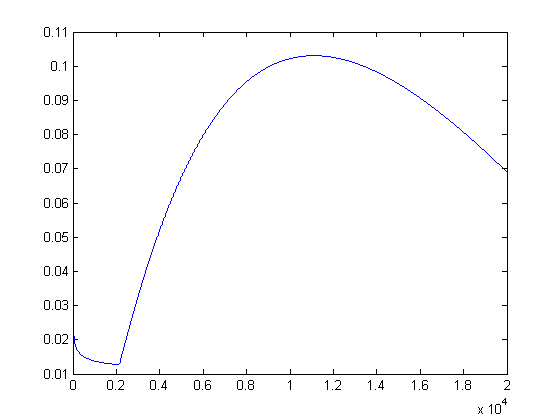}}
\captionof{figure}{The evolution of $\| e_m(p)\|$ as a function of $m$ for $1\leq m\leq 2\times 10^4$.}
\label{fig.bootstrapiterzoomin}
\end{center}

\begin{center}
  \centering
  \centerline{\includegraphics[scale = 0.5]{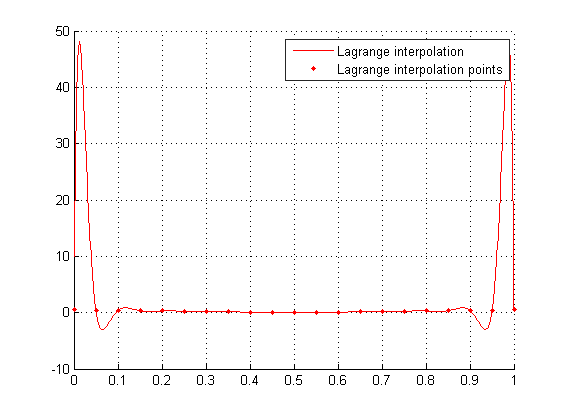}}
\captionof{figure}{The Lagrange interpolation points and the Lagrange interpolation polynomial of the function $|p-1/2|$ with equi-distant $n+1$ points, where $n = 20$. }
\label{fig.lagrangefinalerror}
\end{center}

Figure~\ref{fig.bootstrapiter},~\ref{fig.bootstrapiterzoomin}, and~\ref{fig.lagrangefinalerror} show that the behavior of $\| e_m(p)\|$ could be highly irregular: in fact, for the specific function $f(p) = |p-1/2|$, it continues to decrease until $m$ grows slightly above $2\times 10^3$, and then keeps on increasing until $m$ exceeds about $1.2\times 10^4$, then it continues to drop until it hits about $3\times 10^4$, then it keeps on increasing again within the range of computations we conduct. It is also clear that after about $8.5\times 10^5$ bootstrap iterations, which is by no means practical, $\| e_m(p)\|$ is still far from its limit $\| f- L_n[f]\|$, which is about $47.5945$ as shown in Figure~\ref{fig.lagrangefinalerror}. 

\begin{remark}[Connections between bootstrap and jackknife]
The interested reader must have observed that the bias properties of the bootstrap bias corrected estimator after $r-1$ rounds are the same as that of the $r$-jackknife estimator satisfying Condition~\ref{condi.boundedcoeffcondition}. Concretely, their biases are both dictated by the modulus $\omega_\varphi^{2r}(f,1/\sqrt{n})$. It would be interesting to compare the rate $\omega_\varphi^{2r}(f,1/\sqrt{n})$ with that of the best polynomial approximation, upon noting that in the binomial model, the biases of both the jackknife and bootstrap estimators are polynomial approximation errors of the function $f(p)$ with degree at most $n$. It follows~\cite[Thm. 7.2.1.]{Ditzian--Totik1987} that for best polynomial approximation with degree $n$, the approximation error $\inf_{P \in \poly_n} \sup_{p\in [0,1]} |f(p)- P(p)|$ is upper bounded by $\omega_\varphi^k(f,1/n)$ for \emph{any} $k<n$. We first observe that one achieves a smaller argument ($1/n$ compared to $1/\sqrt{n}$) in this case, but more importantly, there is essentially no restriction on the modulus order when $n$ is large. It indicates the best polynomial approximation induces a much better approximation (smaller bias) for estimating $f(p)$, which, unfortunately has been shown in~\cite{Paninski2003} to fail to achieve the minimax rates in entropy estimation, since the variance explodes while the bias is very small. The estimators in~\cite{Valiant--Valiant2011power,Wu--Yang2014minimax, Jiao--Venkat--Han--Weissman2014maximum} only choose to conduct best polynomial approximation in certain regimes of $f$, which reduces the bias by a logarithmic factor without increasing too much the variance. 
\end{remark}

\subsection{Taylor series bias correction}

The Taylor series can only be applied to functions with certain global differentiablity conditions, which makes it a less versatile method compared to the bootstrap and jackknife. The Taylor series bias correction method exhibits various forms in the literature, and we discuss two of them in this section. We call one approach the \emph{iterative first order correction}, and the other approach the \emph{sample splitting correction}. To illustrate the main ideas behind the methods, we still use the binomial model $n \cdot \hat{p}_n \sim \mathsf{B}(n,p)$. 

\subsubsection{Iterative first order correction}

As shown in~\cite[Chapter 6, Section 1, Pg. 436]{Lehmann--Casella1998theory}, suppose for certain $f$, we have
\begin{align}
\mathbb{E}_p f(\hat{p}_n) - f(p) & = \frac{B_n(p)}{n} + O\left( \frac{1}{n^2} \right),
\end{align}
where $B_n(p) = \frac{1}{2} f''(p) n \mathbb{E}_p (\hat{p}_n - p)^2$. Then, the Taylor series bias corrected estimator is defined as
\begin{align}\label{eqn.orderonetaylorbiascorrection}
\hat{f}_2 & = f(\hat{p}_n) - \frac{B_n(\hat{p}_n)}{n}. 
\end{align}

We can generalize the approach above to conduct bias correction for multiple rounds~\cite{Withers1987}. However, the correction formula becomes increasingly more complicated as the correction order becomes higher. We start with the following lemma.
\begin{lemma}\label{lemma.taylorbiascorrectionbasic}
Suppose function $f: [0,1] \mapsto \mathbb{R}$ satisfies condition $D_s$ with parameter $L$ as in Condition~\ref{condi.conditionds}, where $s = 2k$ is a positive even integer. Then, if $n \cdot \hat{p}_n \sim \mathsf{B}(n,p)$, there exist $k-1$ linear operators denoted as $T_j[f](p),1\leq j\leq k-1$, independent of $n$, such that
\begin{align}
\left | \mathbb{E}_p f(\hat{p}_n) - f(p) - \sum_{j =1}^{k-1} \frac{1}{n^j} T_j[f](p) \right | & \lesssim_{k,L} \frac{1}{n^k}.
\end{align}
Here $\sup_{p\in [0,1]} |T_j[f](p)|  \lesssim_{k,L} 1 $. Concretely, $T_j[f](p)$ is a linear combination of the derivatives of $f$ of order from $j+1$ to $2j$ where the combination coefficients are polynomials of $p$ with degree no more than $2j$.  
\end{lemma}

Now we describe the Taylor series bias correction algorithm below~\cite{Withers1987}. 
\begin{construction}[Taylor series bias correction]~\cite{Withers1987}\label{const.taylorseriesbias}
Define $t_i(p)$ iteratively. Set $t_0(p) = f(p)$, and for $i\geq 1$ define
\begin{align}
t_i(p) & = -\sum_{j = 1}^i T_j[t_{i-j}](p).
\end{align}
The final bias corrected estimator is
\begin{align}
\hat{f}_k & = \sum_{i = 0}^{k-1} \frac{1}{n^i} t_i( \hat{p}_n). 
\end{align}
\end{construction}

Construction~\ref{const.taylorseriesbias} may be intuitively understood as the iterative generalization of the order one Taylor series bias correction~(\ref{eqn.orderonetaylorbiascorrection}). Indeed, after we conduct the first order bias correction and obtain
\begin{align}
\hat{f}_2 & = f(\hat{p}_n) - \frac{T_1[f](\hat{p})}{n} \\
& = t_0(\hat{p}_n) + \frac{t_1(\hat{p}_n)}{n},
\end{align}
we apply Lemma~\ref{lemma.taylorbiascorrectionbasic} to the function $t_0 + \frac{t_1}{n}$ and obtain the expansion up to order $\frac{1}{n^2}$ as
\begin{align}
&t_0(p) + \frac{t_1(p)}{n} + \frac{T_1[t_0](p)}{n} + \frac{T_2[t_0](p)}{n^2} + \frac{T_1[t_1](p)}{n^2} \nonumber \\
& \quad  = t_0(p) +  \frac{T_2[t_0](p)}{n^2} + \frac{T_1[t_1](p)}{n^2},
\end{align}
where we used the definition of $t_1 = -T_1[t_0]$. It naturally leads to the further correction
\begin{align}
t_2(p) = - T_2[t_0](p) - T_1[t_1](p). 
\end{align}
One can repeat this process to obtain the formula in Construction~\ref{const.taylorseriesbias}.

Now we prove that the estimator $\hat{f}_k$ in Construction~\ref{const.taylorseriesbias} achieves bias of order $O(n^{-k})$ if the original function $f$ satisfies condition $D_s$ with $s = 2k$. It can be viewed as one concrete example of~\cite{Withers1987}. 

\begin{theorem}\label{thm.withersbiascorrectionperformance}
Suppose $f: [0,1] \mapsto \mathbb{R}$ satisfies condition $D_s$ with parameter $L$, and $s = 2k, k\geq 1, k\in \mathbb{Z}$. Then, the estimator in Construction~\ref{const.taylorseriesbias} satisfies
\begin{align}
 \| \mathbb{E}_p[\hat{f}_k] - f(p) \| \lesssim_{k,L} \frac{1}{n^k}. 
\end{align} 
\end{theorem}

\subsubsection{Sample splitting correction}
This method was proposed in~\cite{Han--Jiao--Weissman2016minimax}. It aims at solving one disadvantage of Construction~\ref{const.taylorseriesbias}, which is that the bias correction formula for higher orders may not be easy to manipulate since it is defined through a recursive formula. The \emph{sampling splitting correction} method provides an explicit bias correction formula which is easy to analyze with transparent proofs, but the disadvantage it has is that it only applies to certain statistical models. 

The intuition of the sample splitting correction method is the following, which is taken from~\cite{Han--Jiao--Weissman2016minimax}. Suppose $f$ satisfies condition $D_{2k}$ with parameter $L$. Instead of doing Taylor expansion of $f(\hat{p}_n)$ near $p$, we employ Taylor expansion of $f(p)$ near $\hat{p}_n$:
\begin{align}
f(p) \approx \sum_{i = 0}^{2k-1} \frac{f^{(i)}(\hat{p}_n)}{i!} (p - \hat{p}_n)^i. 
\end{align}

Now, $f^{(i)}(\hat{p}_n)$ is by definition an unbiased estimator for $\mathbb{E}_p [f^{(i)}(\hat{p}_n)]$. However, the unknown $p$ in the right hand side still prevents us from using this estimator explicitly. Fortunately, this difficulty can be overcome by the standard sample splitting approach: we split samples to obtain independent $\hat{p}_n^{(1)}$ and $\hat{p}_n^{(2)}$, both of which follow the same class of distribution (with possibly different parameters) as $\hat{p}_n$. We remark that sample splitting can be employed for divisible distributions, including multinomial, Poisson and Gaussian models \cite{nemirovski2000topics}. Now our bias-corrected estimator is
\begin{align}\label{eq.bias_cor}
\hat{f}_k = \sum_{i=0}^{2k-1} \frac{f^{(i)}(\hat{p}_n^{(1)})}{i!}\sum_{j=0}^i \binom{i}{j}S_j(\hat{p}_n^{(2)})(-\hat{p}_n^{(1)})^{i-j}
\end{align}
where $S_j(\hat{p}_n^{(2)})$ is an unbiased estimator of $p^j$ (which usually exists when sample splitting is doable). Now it is straightforward to show that
\begin{align}
\bE[\hat{f}_k] - f(p) & = \bE_p\left[\sum_{i=0}^{2k-1} \frac{f^{(i)}(\hat{p}_n^{(1)})}{i!}(p-\hat{p}_n^{(1)})^i - f(p) \right] \\
& \lesssim \mathbb{E}_p\left|\frac{ \| f^{(2k)} \| }{(2k)!}(\hat{p}_n^{(1)}-p)^{2k}\right| \\
& \lesssim_{k,L} \frac{1}{n^k},
\end{align}
where in the last step we used the property of the binomial distribution in Lemma~\ref{lemma.centralmomentcoeffbounds}. 

The rest of the paper is organized as follows. Section~\ref{sec.jackknifeothersmain} discusses the key proof ingredients of the results pertaining to jackknife bias correction. The proofs of main results on bootstrap bias correction are provided in Section~\ref{sec.bootstrapmainrest}. Appendix~\ref{sec.kfunctional} reviews the $K$-functional approach for bias analysis. Appendix~\ref{sec.auxlemmas} collects auxiliary lemmas used throughout this paper. Proofs of the rest of the theorems and lemmas in the main text are provided in Appendix~\ref{sec.proofmaintheoremslemmas}, and the proofs of the auxiliary lemmas are presented in Appendix~\ref{sec.proofofauxlemmas}.

\section{Jackknife bias correction}\label{sec.jackknifeothersmain}

\subsection{Theorem~\ref{thm.jackknife_general_rs}}

We first present the proof of Theorem~\ref{thm.jackknife_general_rs}. We explain the roadmap below, and the key lemmas used in roadmap are proved in Appendix~\ref{sec.proofmaintheoremslemmas}. 

The first step to analyze the general $r$-jackknife for functions $f$ satisfying the Condition $D_s$ in Condition~\ref{condi.conditionds} is to use Taylor expansions. It is reflected in Lemma~\ref{lemma.Taylorjackknife}.

\begin{lemma}\label{lemma.Taylorjackknife}
Suppose $f$ satisfies Condition~\ref{condi.conditionds} for fixed $s\geq 1$ with parameter $L$. Then, for the general $r$-jackknife estimator with fixed $r\geq 1$ in Definition~\ref{def.generalrjackknifeestimatordefinition}, 
\begin{align}
|\EE_p \hat{f}_r - f(p)| &\lesssim_{r,s,L} n^{-r} \nonumber \\ 
& \qquad + \int_p^1 \left|\sum_{i=1}^r C_i\EE_p(\hat{p}_{n_i}-t)_+^{s-1}\right|dt \nonumber\\
&\qquad + \int_0^p \left|\sum_{i=1}^r C_i\EE_p(\hat{p}_{n_i}-t)_-^{s-1}\right|dt.
\end{align}
Here the coefficients $\{C_i\}_{1\leq i\leq r}$ are given in Definition~\ref{def.generalrjackknifeestimatordefinition}. 
\end{lemma}

Lemma~\ref{lemma.Taylorjackknife} shows that it suffices to analyze the behavior of the quantities $\sum_{i=1}^r C_i\EE_p(\hat{p}_{n_i}-t)_+^{s-1}$ and $\sum_{i=1}^r C_i\EE_p(\hat{p}_{n_i}-t)_-^{s-1}$. These quantities can be viewed as divided differences (see (\ref{eqn.dvdjackknifebias})), and to analyze the worst case we analyze the following backward difference sequences. 

For $t\ge p, u\geq 0, s\geq 0$, define
\begin{align}
A_{n,u}(t) = \EE_p (\hat{p}_n-t)_+^{u}
\end{align}
and consider its $s$-backward difference defined as
\begin{align}
\Delta^s A_{n,u}(t) \triangleq \sum_{k=0}^s (-1)^k\binom{s}{k} A_{n-k,u}(t).
\end{align}
We have
\begin{lemma}\label{lemma.A_nr}
For $s,u\ge 0$, $t\ge p$ and $n\ge 2s$, if $p\le\frac{1}{n}$,
\begin{align}
|\Delta^s A_{n,u}(t)|\le & c_1\cdot
\left(n^{-(u+s-1)}p + n^{-u}p^{u\wedge 1}(\frac{1}{\sqrt{nt}}\wedge 1)\right)\nonumber \\ 
& \cdot \exp(-c_2nt);
\end{align}
if $\frac{1}{n}<p\le \frac{1}{2}$, 
\begin{align}
    |\Delta^s A_{n,u}(t)|\le & c_1\cdot
\left(n^{-(\frac{u}{2}+s)}p^{\frac{u}{2}} + n^{-u}p^{u\wedge 1}(\frac{1}{\sqrt{nt}}\wedge 1)\right)\nonumber \\ 
& \cdot \exp(-\frac{c_2n(t-p)^2}{t});
\end{align}
if $\frac{1}{2}<p\le 1-\frac{1}{n} $,
\begin{align}
    |\Delta^s A_{n,u}(t)|\le & c_1\cdot
\Big(n^{-(\frac{u}{2}+s)}(1-p)^{\frac{u}{2}} + n^{-u}(1-p)^{u\wedge 1} \nonumber \\ 
& \cdot (\frac{1}{\sqrt{n(1-t)}}\wedge 1)\Big) \cdot \exp(-\frac{c_2n(t-p)^2}{1-p}),
\end{align}
where the universal constants $c_1,c_2>0$ only depend on $u,s$ (not on $n$ or $p$). Moreover, if $t>1-\frac{1}{n}$, we have
\begin{align}
|\Delta^s A_{n,u}(t)|\le c_1(1-t)^u(1-p)^sp^{n-s}.
\end{align}
\end{lemma}
Note that in Lemma \ref{lemma.A_nr}, the case where $p>1-\frac{1}{n}$ has already been included in the case $t>1-\frac{1}{n}$, for $t\ge p$. Hence, Lemma \ref{lemma.A_nr} has completely characterized an upper bound on the dependence of $|\Delta^s A_{n,u}(t)|$ on all $n, p$ and $t$. By symmetry, we have the following corollary regarding
\begin{align}
A_{n,u}^-(t) \triangleq \EE_p(\hat{p}_n-t)_-^u.
\end{align}
\begin{corollary}\label{cor.A_nr}
For $s,u\ge 0$, $t\le p$ and $n\ge 2s$, if $p\ge 1-\frac{1}{n}$,
\begin{align}
    |\Delta^s A_{n,u}^-(t)|\le&  c_1\cdot
\Bigg(n^{-(u+s-1)}(1-p) + n^{-u}(1-p)^{u\wedge 1}\nonumber \\ 
& \cdot (\frac{1}{\sqrt{n(1-t)}}\wedge 1)\Bigg)\cdot \exp(-c_2n(1-t));
\end{align}
if $\frac{1}{2}\le p\le 1-\frac{1}{n}$,
\begin{align}
    |\Delta^s A_{n,u}^-(t)|\le &c_1\cdot
\Bigg(n^{-(\frac{u}{2}+s)}(1-p)^{\frac{u}{2}} + n^{-u}(1-p)^{u\wedge 1}\nonumber \\ 
& \cdot (\frac{1}{\sqrt{n(1-t)}}\wedge 1)\Bigg)\cdot \exp(-\frac{c_2n(t-p)^2}{1-t});
\end{align}
if $\frac{1}{n}\le p<\frac{1}{2}$,
\begin{align}
    |\Delta^s A_{n,u}^-(t)|\le & c_1\cdot
\left(n^{-(\frac{u}{2}+s)}p^{\frac{u}{2}} + n^{-u}p^{u\wedge 1}(\frac{1}{\sqrt{nt}}\wedge 1)\right) \nonumber \\ & \cdot \exp(-\frac{c_2n(t-p)^2}{p});
\end{align}
where the universal constants $c_1,c_2>0$ only depend on $u,s$ (not on $n$ or $p$). Moreover, if $t<\frac{1}{n}$, we have
\begin{align}
|\Delta^s A_{n,u}(t)|\le c_1t^up^s(1-p)^{n-s}.
\end{align}
\end{corollary}

Furthermore, in most cases we do not need the dependence on $p$, and Lemma \ref{lemma.A_nr} implies the following corollary.
\begin{corollary}\label{cor.A_nr_onlyn}
For $s,u\ge 0$, $t\ge p\ge t'$ and $n\ge 2s$, we have
\begin{align}
&|\Delta^s A_{n,u}(t)|\le c_1\big(n^{-(\frac{u}{2}+s)}\exp(-c_2n(t-p)^2)\nonumber \\
&\quad +  n^{-(u+\frac{1}{2})}\cdot \frac{1}{\sqrt{t+n^{-1}}}\exp(-\frac{c_2n(t-p)^2}{t}) \big)\\
&|\Delta^s A_{n,u}^-(t')|\le c_1\big(n^{-(\frac{u}{2}+s)}\exp(-c_2n(t'-p)^2) \nonumber \\ 
&\quad +  n^{-(u+\frac{1}{2})}\cdot \frac{1}{\sqrt{1-t'+n^{-1}}}\exp(-\frac{c_2n(t'-p)^2}{1-t}) \big)
\end{align}
where the universal constants $c_1,c_2>0$ only depend on $u,s$ (not on $n$ or $p$).
\end{corollary}

Now we can start the proof of Theorem~\ref{thm.jackknife_general_rs}. 
\begin{proof}[Proof of Theorem~\ref{thm.jackknife_general_rs}]
We split into three cases. When $s=0$, we use the triangle inequality to conclude that
\begin{align}
|\EE_p \hat{f}_r(\hat{p}_n)-f(p)| &= \left|\sum_{i=1}^r C_i\EE_pf(\hat{p}_{n_i})-f(p)\right|\\
&\le (\sum_{i=1}^r |C_i|+1) \|f\|\\
&\lesssim_L n^{r-1}.
\end{align}

For $1\le s\le 2r$, it follows from Lemma~\ref{lemma.Taylorjackknife} that 
\begin{align}
& |\EE_p \hat{f}_r(\hat{p}_n) - f(p)| \nonumber\\ \lesssim_{r,s,L} & n^{-r} + \int_p^1 \left|\sum_{i=1}^r C_i\EE_p(\hat{p}_{n_i}-t)_+^{s-1}\right|dt \nonumber \\ &
+ \int_0^p \left|\sum_{i=1}^r C_i\EE_p(\hat{p}_{n_i}-t)_-^{s-1}\right|dt.
\end{align}

For $t\ge p$, it follows from Corollary \ref{cor.A_nr_onlyn} that there exist universal constants $c_1,c_2$ depending on $r,s$ only such that
\begin{align}
|\Delta^uA_{n,s-1}(t)|&\le c_1(\frac{1}{\sqrt{t+n^{-1}}}e^{-\frac{c_2n(t-p)^2}{t}}+e^{-c_2n(t-p)^2})\nonumber\\
&\qquad\cdot \begin{cases}
n^{-(\frac{s-1}{2}+u)} &\text{if }0\le u\le \lfloor \frac{s}{2}\rfloor,\\
n^{-(s-\frac{1}{2})} &\text{if }u\ge \lceil \frac{s}{2}\rceil.
\end{cases}
\end{align}

Now define
\begin{align}
B_{n,r,s}(t) = n^{r-1}A_{n,s}(t). 
\end{align}

By the product rule of backward difference we obtain
\begin{align}
& |\Delta^{r-1}B_{n,r,s}(t)| \nonumber \\ 
\lesssim_{r,s} & \sum_{0\le i,j\le r-1,i+j\ge r-1} |\Delta^i n^{r-1}|\cdot |\Delta^j A_{n,s}(t)|\\
\lesssim_{r,s} &\sum_{0\le i,j\le r-1,i+j\ge r-1} n^{r-1-i}\cdot (\frac{1}{\sqrt{t+n^{-1}}}e^{-\frac{c_2n(t-p)^2}{t}}\nonumber \\ 
& \qquad \qquad \qquad +e^{-c_2n(t-p)^2}) \cdot n^{-\min\{\frac{s-1}{2}+j, s-\frac{1}{2}\}}\\
\lesssim_{r,s} & (\frac{1}{\sqrt{t+n^{-1}}}e^{-\frac{c_2n(t-p)^2}{t}}+e^{-c_2n(t-p)^2})\nonumber \\ 
& \cdot n^{-\min\{\frac{s-1}{2},s-r+\frac{1}{2}\}}.
\end{align}

As a result of Lemma~\ref{lemma.meanvaluetheorem},
\begin{align}
&\quad \left|\sum_{i=1}^r C_i\EE_p(\hat{p}_{n_i}-t)_+^{s-1}\right| \nonumber \\ 
&=  \left|\sum_{i=1}^r C_iA_{n_i,s-1}(t)\right|  \\
&= \left|B_{\cdot,r,s}[n_1,\cdots,n_r](t)\right| \\
&\le \frac{1}{(r-1)!}\max_{m\in [n_1,n_r]} |\Delta^r B_{m,r,s}(t)|\\
&\lesssim_{r,s}  (\frac{1}{\sqrt{t+n^{-1}}}e^{-\frac{c_2n(t-p)^2}{t}}+e^{-c_2n(t-p)^2})\nonumber \\ &  \quad \cdot n^{-\min\{\frac{s-1}{2},s-r+\frac{1}{2}\}}.
\end{align}
Using this inequality, finally we arrive at
\begin{align}
&\int_p^1 \left | \sum_{i = 1}^r C_i \mathbb{E}_p (\hat{p}_{n_i}-t)_{+}^{s-1} \right | dt \nonumber\\
&\lesssim_{r,s} \int_p^1 (\frac{1}{\sqrt{t+n^{-1}}}e^{-\frac{c_2n(t-p)^2}{t}}+e^{-c_2n(t-p)^2})\nonumber\\ 
& \quad \cdot n^{-\min\{\frac{s-1}{2},s-r+\frac{1}{2}\}}dt\\
&\le n^{-\min\{\frac{s-1}{2},s-r+\frac{1}{2}\}}\big(\int_0^\infty\frac{1}{\sqrt{u+p+n^{-1}}}\nonumber \\ & \qquad \cdot e^{-\frac{c_2nu^2}{u+p}}du  +\int_0^\infty e^{-c_2nu^2}du \big)\\
&\le n^{-\min\{\frac{s-1}{2},s-r+\frac{1}{2}\}}\big(\int_0^p \frac{1}{\sqrt{p}}e^{-\frac{c_2nu^2}{2p}}du \nonumber \\ & \qquad +\int_p^\infty \sqrt{n}e^{-\frac{c_2nu}{2}}du + \int_0^\infty e^{-c_2nu^2}du \big)\\
&\le n^{-\min\{\frac{s-1}{2},s-r+\frac{1}{2}\}}\big(\frac{1}{\sqrt{p}}\int_0^\infty e^{-\frac{c_2nu^2}{2p}}du \nonumber \\ & \qquad +\sqrt{n}\int_0^\infty e^{-\frac{c_2nu}{2}}du+ \int_0^\infty e^{-c_2nu^2}du \big)\\
&\lesssim_{r,s} n^{-\min\{\frac{s}{2},s-r+1\}}
\end{align}
as desired. The remaining part can be dealt with analogously.

When $s\ge 2r$, the desired result follows from applying Lemma~\ref{lemma.Taylorjackknife} with $s = 2r$. 
\end{proof}

\subsection{Theorem~\ref{thm.jackknife_rs_converse}}

We consider the case where $s>0$ and $s\le 2r-3$. To come up with an example which matches the upper bound in Theorem \ref{thm.jackknife_general_rs}, we first need to prove a ``converse" of Lemma \ref{lemma.A_nr}. Recall that for $t\ge p$, 
\begin{align}
A_{n,u}(t) = \EE_p (\hat{p}_n-t)_+^u.
\end{align}
\begin{lemma}\label{lem.A_nr_converse}
For any $0\le u\le 2(s-1)$, there exists some $p_0>0$ such that for any $0<p<\min\{p_0,\frac{1}{4s}\}$ and any $n\ge \frac{1}{p^2}$, whenever $t\in [p,p+\frac{1}{\sqrt{n}}]$ satisfies 
\begin{enumerate}
\item $\frac{k}{n-s+1}<t<\frac{k-p}{n-s}$ for some $k\in\mathbb{N}$ if $u<s$;
\item $\frac{k}{n-s}<t<\frac{k+p}{n-s}$ for some $k\in\mathbb{N}$ if $u\ge s$,
\end{enumerate}
we have
\begin{align}
|\Delta^s A_{n,u}(t)| \ge cn^{-(u+\frac{1}{2})}
\end{align}
where $c>0$ is a universal constant which only depends on $u,s$ and $p$.
\end{lemma}

Now we start the proof of Theorem~\ref{thm.jackknife_rs_converse}. The basic idea of the proof is to construct functions $f$ such that Lemma~\ref{lemma.Taylorjackknife} is nearly tight. 
\begin{proof}[Proof of Theorem~\ref{thm.jackknife_rs_converse}]
Pick an arbitrary $p>0$ which satisfies Lemma \ref{lem.A_nr_converse}, and we define
\begin{align}
g(t) = \text{sign}\left(\sum_{i=1}^r C_i\EE_p(\hat{p}_{n_i}-t)_+^{s-1}\right)\cdot \mathbbm{1}(t\ge p).
\end{align}
Note that $g$ is not continuous, but we can find some $h\in C[0,1]$ such that $\|g-h\|_1\le n^{-2r}$. Now choose any $f$ with $f^{(s)}=h$ , we know that $f\in C^{s}[0,1]$, and the norm conditions are satisfied under proper scaling.

It follows from Lemma~\ref{lemma.Taylorjackknife} that
\begin{align}
&|\EE_p\hat{f}_r(\hat{p}_n)-f(p)| \nonumber \\ 
\asymp & O(n^{-r}) + \int_p^1 h(t)\left(\sum_{i=1}^r C_i\EE_p(\hat{p}_{n_i}-t)_+^{s-1}\right)dt 
 \nonumber \\ 
 & + \int_0^p h(t)\left(\sum_{i=1}^r C_i\EE_p(\hat{p}_{n_i}-t)_-^{s-1}\right)dt\\
= &O(n^{-r}) + \int_0^1 g(t)\left(\sum_{i=1}^r C_i\EE_p(\hat{p}_{n_i}-t)_+^{s-1}\right)dt\nonumber \\ 
 &  + O(\|g-h\|_1)\cdot \sum_{i=1}^{r}|C_i|\\
=& O(n^{-r})+ \int_p^1 \left|\sum_{i=1}^r C_i\EE_p(\hat{p}_{n_i}-t)_+^{s-1}\right|dt\\
= &o(n^{r-1-s})+ \int_p^1 \left|\sum_{i=1}^r C_iA_{n_i,s-1}(t)\right|dt\\
\ge & o(n^{r-1-s})+ \int_G \left|\sum_{i=1}^r C_iA_{n_i,s-1}(t)\right|dt
\end{align}
where $G\subset [p,p+\frac{1}{\sqrt{n}}]$ is the set of all ``good" $t$'s which satisfy the condition of Lemma \ref{lem.A_nr_converse}. It's easy to see
\begin{align}
m(G) \asymp n^{-\frac{1}{2}}
\end{align}
where $m(\cdot)$ denotes the Lebesgue measure. Moreover, by our choice of delete-1 jackknife, for $t\in G$ we have
\begin{align}
& \left|\sum_{i=1}^r C_iA_{n_i,s-1}(t)\right|\nonumber \\ 
  = &\left|\Delta^{r-1} (n^{r-1}A_{n,s-1}(t))\right|\\
\gtrsim & n^{r-1}|\Delta^{r-1}A_{n,s-1}(t)| \nonumber\\
& - \sum_{1\le i\le r-1,0\le j\le r-1,i+j\ge r-1} |\Delta^{i}n^{r-1}|\cdot |\Delta^jA_{n,s-1}(t)|\\
\gtrsim & n^{r-1-(s-\frac{1}{2})} - \sum_{\substack{1\le i\le r-1,0\le j\le r-1, \\i+j\ge r-1}} n^{r-1-i-\min\{\frac{s-1}{2}+j,s-\frac{1}{2}\}}\\
\gtrsim & n^{r-s-\frac{1}{2}}
\end{align}
where we have used Lemma \ref{lemma.A_nr}, Lemma \ref{lem.A_nr_converse} and the assumption that $0\le s-1\le 2(r-2)$. As a result, we conclude that
\begin{align}
&|\EE_p\hat{f}_r(\hat{p}_n)-f(p)| \nonumber \\  \gtrsim &o(n^{r-1-s})+ \int_G \left|\sum_{i=1}^r C_iA_{n_i,s-1}(t)\right|dt\\
\gtrsim& o(n^{r-1-s}) + m(G)\cdot  n^{r-s-\frac{1}{2}} \\
\gtrsim & n^{r-s-1}
\end{align}
as desired.
\end{proof}

\subsection{Theorem~\ref{thm.specialfunctionslikeentropy}}

The following lemma characterizes the difference $\mathbb{E}_p f(\hat{p}_n) - \mathbb{E}_p f(\hat{p}_{n-1})$ for certain functions. 
\begin{lemma}\label{lemma.specialfunctionstiman}\cite{strukov1977mathematical}
Suppose $f(p) = -p\ln p, p\in [0,1]$. Then, 
\begin{align}
0 \leq \mathbb{E}_p f(\hat{p}_n) - \mathbb{E}_p f(\hat{p}_{n-1}) \leq \frac{1- p^n - (1-p)^n}{n(n-1)}. 
\end{align}
Suppose $f(p) = p^\alpha, p \in [0,1], 0<\alpha<1$. Then, 
\begin{align}
0 \leq  & \mathbb{E}_p f(\hat{p}_n) - \mathbb{E}_p f(\hat{p}_{n-1}) \nonumber \\ 
\leq &  \frac{(1-\alpha)(1-p^n - (1-p)^n)}{n(n-1)^\alpha}. 
\end{align}
\end{lemma}

The next lemma characterizes the lower bound for the bias of the jackknife estimate $\hat{f}_2$. 
\begin{lemma}\label{lemma.jackknifelowerboundspeicalfunctions}
Suppose $\hat{f}_2$ is the delete-$1$ $2$-jackknife. Then,
\begin{enumerate}
\item for $f(p) = -p\ln p$, 
\begin{align}
\| \mathbb{E}_p \hat{f}_2 - f(p) \| \gtrsim \frac{1}{n}.
\end{align}
\item for $f(p) = p^\alpha, 0<\alpha<1$, 
\begin{align}
\| \mathbb{E}_p \hat{f}_2 - f(p) \| \gtrsim \frac{1}{n^\alpha}.
\end{align}
\end{enumerate}

Suppose $\hat{f}_2$ is a $2$-jackknife that satisfies Condition~\ref{condi.boundedcoeffcondition}. Then, 
\begin{enumerate}
\item for $f(p) = -p\ln p$, 
\begin{align}
\| \mathbb{E}_p \hat{f}_2 - f(p) \| \gtrsim \frac{1}{n}.
\end{align}
\item for $f(p) = p^\alpha, 0<\alpha<1$, 
\begin{align}
\| \mathbb{E}_p \hat{f}_2 - f(p) \| \gtrsim \frac{1}{n^\alpha}.
\end{align}
\end{enumerate}
\end{lemma}

Now we can start the proof of Theorem~\ref{thm.specialfunctionslikeentropy}. 
\begin{proof}[Proof of Theorem~\ref{thm.specialfunctionslikeentropy}]
The lower bounds follow from Lemma~\ref{lemma.jackknifelowerboundspeicalfunctions}. Now we prove the upper bounds. 

For a general $2$-jackknife and general function $f$, we have
\begin{align}
\mathbb{E}_p [\hat{f}_2] - f(p) & = \frac{n_1}{n_1 - n_2} \left( \mathbb{E}_p[f(\hat{p}_{n_1})] - f(p) \right) \nonumber \\ 
& \quad + \frac{n_2}{n_2 - n_1} \left( \mathbb{E}_p [f(\hat{p}_{n_2})] - f(p) \right). 
\end{align}

Define $H_n = \mathbb{E}_p[f(\hat{p}_n)] -f(p)$. Then,
\begin{align}
\mathbb{E}_p [\hat{f}_2] - f(p)  & = \frac{n_2 H_{n_2} - n_1 H_{n_1}}{n_2 - n_1} \\
& = H_{n_2} + \frac{n_1}{n_2-n_1} \left( H_{n_2} - H_{n_1} \right). 
\end{align}

For any $f\in C[0,1]$, we have $\lim_{n\to \infty} H_n = 0$, which implies
\begin{align}
H_{n_2} & = H_{n_2} - H_\infty \\
& = \sum_{j = n_2}^\infty \left( H_j - H_{j+1} \right)
\end{align}
and 
\begin{align}
H_{n_2} -H_{n_1} & = \sum_{j = n_1+1}^{n_2} \left( H_{j} - H_{j-1} \right) 
\end{align}
It follows from Lemma~\ref{lemma.jackknifelowerboundspeicalfunctions} that for $f(p) = -p\ln p$, $0\leq H_{j} - H_{j-1} \lesssim \frac{1}{j^2}$. Hence, 
\begin{align}
\| \mathbb{E}_p [\hat{f}_2] - f(p) \| & \lesssim \sum_{j = n_2}^\infty \frac{1}{j^2} +  \frac{n_1}{n_2 - n_1} \cdot \left(\sum_{j = n_1+1}^{n_2} \frac{1}{j^2} \right) \\
& \lesssim \frac{1}{n_2} + \frac{n_1}{n_2 - n_1} \left( \frac{1}{n_1} - \frac{1}{n_2} \right) \\
& \lesssim \frac{1}{n_2} \\
& \lesssim \frac{1}{n},
\end{align}
where in the last step we used Definition~\ref{def.generalrjackknifeestimatordefinition}. 

The case of $f(p) = p^\alpha, 0<\alpha<1$ can be proved analogously. 
%
%

%
%
\end{proof}

\section{Bootstrap bias correction} \label{sec.bootstrapmainrest}

\subsection{Theorem~\ref{thm.bootstraplimit}}

Define the Bernstein operator $B_n: C[0,1] \mapsto C[0,1]$ as 
\begin{align}
B_n[f](p) = \sum_{i = 0}^n f\left(\frac{i}{n} \right) {n \choose i}  p^i (1-p)^{n-i}. 
\end{align}

Theorem~\ref{thm.bootstraplimit} can be proved using the eigenstructure of the Bernstein operator~\cite{sevy1995lagrange}. We give a concrete proof as follows. 

It was shown in~\cite{cooper2000eigenstructure} that the Bernstein operator $B_n[f]$ admits a clean eigenstructure. Concretely, it has $n+1$ linearly independent eigenfunctions $p_k^{(n)}, 0\leq k \leq n$, which are polynomials with order $k$, with $k$ simple zeros on $[0,1]$. The corresponding eigenvalues are $\lambda_k^{(n)} = \frac{1}{n^k} \frac{n!}{(n-k)!}$. The Bernstein operator is also degree reducing in the sense that it maps a $k$-degree polynomial to another polynomial with degree no more than $k$. 

Decomposing $f(p)$ as $f(p) = L_n[f](p) + g(p)$. It follows the definition of $L_n[f]$ that $g(i/n) = 0, 0\leq i \leq n$. Hence, $B_n[g] \equiv 0$. 

Since $L_n[f](p)$ is a polynomial with degree no more than $n$, it admits a unique expansion 
\begin{align}
L_n[f](p) = \sum_{k = 0}^n a_k p_k^{(n)}. 
\end{align}

Applying $I-B_n$ on the decomposition of $f$, we have
\begin{align}
(I-B_n)[f] & = f - B_n[f] \\
& = \sum_{k = 0}^n a_k p_k^{(n)} + g - \sum_{k = 0}^n a_k \lambda_k^{(n)} p_k^{(n)}\\
& = \sum_{k = 0}^n a_k (1-\lambda_k^{(n)}) p_k^{(n)} + g
\end{align}

It follows by induction that 
\begin{align}
e_m(p) = \sum_{k = 0}^n a_k (1-\lambda_k^{(n)})^{m} p_k^{(n)} + g
\end{align}

We have
\begin{align}
\sup_p |e_m(p) - g| & \leq \sum_{k = 0}^n  (1-\lambda_k^{(n)})^{m} |a_k| \sup_p |p_k^{(n)}| \\
& \to 0\quad \text{ as }m \to \infty. 
\end{align}

Note that $\lambda_k^{(n)} = 1 \cdot \left( 1- \frac{1}{n} \right) \ldots  \left( 1- \frac{k-1}{n} \right), k\geq 2$, and $\lambda_k^{(n)} = 1$ when $k = 0,1$. Hence, the smallest $\lambda_k^{(n)}$, which is the slowest to vanish corresponds to $k = n$. Since
\begin{align}
\lambda_n^{(n)} = \frac{n!}{n^n} \approx \frac{\sqrt{2\pi n}}{e^n},
\end{align}
it takes $m \gtrsim \frac{e^n}{\sqrt{n}}$ rounds of iteration to make $(1-\lambda_n^{(n)})^m$ vanish, which is a prohibitively large number in practice.

\section{Acknowledgment}

We are grateful to Yihong Wu for pointing out the the connection of Theorem~\ref{thm.boundentropy} with $f$-divergence inequalities~\cite{sason2016f}. 

\appendices

\section{The $K$-functional approach to bias analysis}\label{sec.kfunctional}

We introduce the $r$-th modulus of smoothness of a function $f:[0,1]\mapsto \mathbb{R}$ as
\begin{align}
\omega^r(f,t) = \sup_{0<h\leq t} \| \Delta^r_{h} f \|,
\end{align}
where $\Delta^r_h f$ is defined in~(\ref{eqn.rsymmetricdifferencedef}).  The $r$-th Ditzian--Totik modulus of smoothness of a function $f: [0,1] \mapsto \mathbb{R}$ is defined in~(\ref{eqn.dtmodulusdefinition}).

Intuitively, the smoother the function is, the smaller is its moduli of smoothness. For $f\in C[0,1]$, we define the $K$-functional $K_r(f,t^r)$ as follows:
\begin{align}
K_r(f,t^r) = & \inf_g \{ \| f - g \| + t^r \|  g^{(r)} \|: \nonumber \\ 
& \quad  g^{(r-1)} \in \text{A.C.}_{\text{loc}} \},
\end{align}
and the $K$-functional $K_{r,\varphi}(f, t^r)$ as
\begin{align}
K_{r,\varphi}(f, t^r) = & \inf_g \{ \| f - g \| + t^r \| \varphi^r g^{(r)} \|: \nonumber \\ 
& \quad g^{(r-1)} \in \text{A.C.}_{\text{loc}} \},
\end{align}
where $g^{(r-1)} \in \text{A.C.}_{\text{loc}}$ means that $g$ is $r-1$ times differentiable and $g^{(r-1)}$ is absolutely continuous in every closed finite interval $[c,d] \subset (0,1)$.  

The remarkable fact is, the $K$-functionals are equivalent to the corresponding moduli of smoothness for \emph{any} function $f\in C[0,1]$. Concretely, we have the following lemma:

\begin{lemma}\cite[Chap. 6, Thm. 2.4, Thm. 6.2]{devore1993constructive}\cite[Thm. 2.1.1.]{Ditzian--Totik1987}\label{lemma.kfunctionalequiva}
There exist constants $c_1>0,c_2>0$ which depend only on $r$ such that for all $f\in C[0,1]$, 
\begin{align}
c_1 \omega^r(f,t) & \leq K_r(f,t^r)  \leq c_2 \omega^r(f,t) \quad \text{for all }t>0\\
c_1 \omega^r_{\varphi}(f,t) & \leq K_{r,\varphi}(f,t^r)  \leq c_2 \omega^r_{\varphi}(f,t) \quad \text{for all }t\leq t_0,
\end{align}
where $t_0>0$ is a constant that only depends on $r$. 
\end{lemma}

We emphasize that the $K$-functionals and moduli of smoothness introduced here are tailored for the interval $[0,1]$ and the supremum norm, which can be generalized to general finite intervals and infinite intervals, and $L_p$ norms. The corresponding equivalence results also hold in those settings. We refer the interested readers to~\cite[Chap. 6]{devore1993constructive} and~\cite[Chap. 2]{Ditzian--Totik1987} for details. For other use of $K$-functionals in statistics and machine learning, we refer the readers to the theory of Besov spaces as interpolation spaces~\cite{hardle2012wavelets} and distribution testing~\cite{blais2016alice}. 

Now we illustrate the $K$-functional approach to bias analysis, which is well known in the approximation theory literature, see, e.g.~\cite{totik1988uniform}. Suppose $X$ is a random variable taking values in $[0,1]$, and we would like to bound the quantity $|\mathbb{E}[ f(X)] - f(\mathbb{E}[X])|$ for any $f\in C[0,1]$. Clearly, $f$ may not be differentiable, so we introduce another function $g\in C[0,1]$ such that $g^{(1)} \in \text{A.C.}_{\text{loc}}$. We proceed as follows:
\begin{align}
&|\mathbb{E} [f(X)] - f(\mathbb{E}[X])| \nonumber \\ 
 = &|\mathbb{E} \big[ f(X) -g(X) +g(X) - g(\mathbb{E}X)\nonumber \\
 & + g(\mathbb{E}X) - f(\mathbb{E}X) \big] | \\
 \leq &2 \| f - g \| + | \mathbb{E}[ g(X) - g(\mathbb{E}X)]| \\
 \leq &2 \| f -g \| + \frac{1}{2} \| g'' \| \mathsf{Var}(X) \\
 = &2 \left( \|f - g\| + t^2 \| g'' \| \right),
\end{align} 
where $t^2 = \frac{\mathsf{Var}(X)}{4}$. Since $g^{(1)} \in \text{A.C.}_{\text{loc}}$ is arbitrary, we know
\begin{align}
|\mathbb{E}[ f(X)] - f(\mathbb{E}[X])| & \leq 2 \inf_g \left\{ \|f - g\| + t^2 \| g'' \| \right\} \\
& = 2 K_2(f, t^2) \\
& = 2 K_2 \left(f, \frac{\mathsf{Var}(X)}{4} \right) \\
& \leq 2c_2 \omega^2\left(f, \frac{\sqrt{\mathsf{Var}(X)}}{2} \right),
\end{align}
where the constant $c_2$ is introduced in Lemma~\ref{lemma.kfunctionalequiva}. 

It was shown in \cite[Chap. 2, Sec. 9, Example 1]{devore1993constructive} that if $f(x) = x\ln x, x\in [0,1]$, then $\omega^2(f,t)  \leq 2(\ln 2)t$. Hence, we have shown that for any random variable $X\in [0,1]$, 
\begin{align}
|\mathbb{E}[X \ln X] - \mathbb{E}[X]\ln(\mathbb{E}[X]) | & \lesssim \sqrt{\mathsf{Var}(X)}. \label{eqn.entropyjensengap}
\end{align}

Specializing to the case where $X = \hat{p}_n$, where $n\cdot \hat{p}_n \sim \mathsf{B}(n,p)$, we have proved that for $f(p) = -p\ln p$, we have
\begin{align}
| \mathbb{E}[f(\hat{p}_n)] - p\ln p | & \lesssim \sqrt{\frac{p(1-p)}{n}}. 
\end{align}

The upper bound $\sqrt{\frac{p(1-p)}{n}}$ is a pointwise bound that becomes smaller when $p$ is close to $0$ or $1$. When $p$ lies in the middle of the interval $[0,1]$, say $p \approx \frac{1}{2}$, the bound is of scale $\frac{1}{\sqrt{n}}$. We now show that using the $K$-functional $K_{r,\varphi}$ instead of $K_r$ results in a better uniform bound in this case, which is of order $\frac{1}{n}$. \footnote{One remarkable fact is that, the $K$-functional approach with $K_{r,\varphi}$ in bias analysis provides the \emph{tight} norm bound for \emph{any} $f\in C[0,1]$ under the binomial model~\cite{Totik1994}. } 

For any $f\in C[0,1]$, $n\cdot \hat{p}_n \sim \mathsf{B}(n,p)$, and any $g\in C[0,1]$ such that $g^{(1)} \in \text{A.C.}_{\text{loc}}$, we have
\begin{align}
&| \mathbb{E}_p[f(\hat{p}_n)] - f(p) | \nonumber \\  \leq & |\mathbb{E}_p[ f(\hat{p}_n) - g(\hat{p}_n) + g(\hat{p}_n) - g(p) + g(p) -f(p) | \\
= & 2 \| f - g \| + | \mathbb{E}_p g(\hat{p}_n) - g(p) | \\
 = &2 \| f - g \| + \Bigg| \mathbb{E}_p \Big[ g'(p)(\hat{p}_n -p)  \nonumber \\ 
 & +\int_p^{\hat{p}_n} (\hat{p}_n - t) g''(t)dt \Big]\Bigg | \\
 = &2 \| f - g\| + \left | \mathbb{E}_p \left[ \int_p^{\hat{p}_n} (\hat{p}_n - t) g''(t)dt \right]  \right | \\
 \leq & 2 \| f - g\| + \mathbb{E}_p \left| \int_p^{\hat{p}_n} \left|  \frac{\hat{p}_n -t}{t(1-t)} \right|  \cdot \left| t(1-t) g''(t) \right | dt \right|  \\
 = & 2 \| f - g\| + \mathbb{E}_p \left| \int_p^{\hat{p}_n} \frac{|\hat{p}_n-p|}{p(1-p)} |t(1-t) g''(t)| dt \right | \\
 \leq & 2 \| f -g \| + \| \varphi^2 g''\| \mathbb{E}_p \left| \frac{(\hat{p}_n - p)^2}{p(1-p)} \right| \\
 \leq & 2 \| f -g \| + \| \varphi^2 g''\| \frac{1}{n},
\end{align}
where $\varphi(t)=\sqrt{t(1-t)}$, and we used the elementary inequality that $\frac{|\hat{p}_n-t|}{t(1-t)} \leq \frac{|\hat{p}_n - p|}{p(1-p)}$ for any $t$ between $p$ and $\hat{p}_n$. Taking the infimum over all $g$, we have
\begin{align}
| \mathbb{E}_p[f(\hat{p}_n)] - f(p) | & \leq 2 K_{r,\varphi}(f, \frac{1}{2n}) \\
& \lesssim \omega_\varphi^2(f, \frac{1}{\sqrt{2n}}).
\end{align}
It follows from~\cite{Jiao--Venkat--Han--Weissman2014maximum} that for $f(p) = -p\ln p$, $\omega_\varphi^2(f,t) \asymp t^2$, which implies that
\begin{align}
| \mathbb{E}_p [\hat{p}_n \ln \hat{p}_n] - p\ln p | \lesssim \frac{1}{n}. 
\end{align}

The functional $\text{Ent}(X) \triangleq \mathbb{E}[X \ln X] - \mathbb{E}[X]\ln(\mathbb{E}[X])$, which is also called \emph{entropy}, plays a crucial role in the theory of concentration inequalities. Concretely, the Herbst argument~\cite{boucheron2013concentration} shows that if $\text{Ent}(e^{\lambda f(X)}) \leq \frac{\lambda^2}{2} \mathbb{E} [ e^{\lambda f(X)}]$, we have sub-Gaussian type concentration $\mathbb{P}\left( f(X) - \mathbb{E}[f(X)] \geq t \right) \leq e^{-t^2/2}$. Due to the significance of the functional $\text{Ent}(X)$, we now present a theorem providing upper and lower bounds of $\text{Ent}(X)$. The key idea in the following proof is to relate the $\text{Ent}(X)$ functional to the KL divergence, whose functional inequalities have been well studied in the literature. Conceivably, they are stronger bounds than those obtained using the general $K$-functional approach~(Lemma~\ref{lemma.strukov} in Appendix~\ref{sec.auxlemmas}).

\begin{theorem}\label{thm.boundentropy}
Suppose $X$ is a non-negative random variable. Denote $\text{Ent}(X) = \mathbb{E}[X \ln X] - \mathbb{E}[X]\ln(\mathbb{E}[X])$. Then, 
\begin{align}
\text{Ent}(X)  &\leq \mathbb{E}[X] \ln \left( 1 + \frac{\mathsf{Var}(X)}{(\mathbb{E}[X])^2} \right) \leq \sqrt{\mathsf{Var}(X)} \\
\text{Ent}(X) & \geq 2 \left( \mathbb{E}[X] - \sqrt{\mathbb{E}[X]} \mathbb{E}[\sqrt{X}] \right) \geq \mathsf{Var}(\sqrt{X}) \\
\text{Ent}(X) & \geq \frac{1}{2}\mathbb{E}[X] \left( \mathbb{E} \left | \frac{X}{\mathbb{E}[X]} -1 \right | \right)^2. 
\end{align}
\end{theorem}

\begin{remark}
It was shown in~\cite{latala2000between} that $\text{Ent}(X) \geq \mathsf{Var}(\sqrt{X})$. Theorem~\ref{thm.boundentropy} strengthens~\cite{latala2000between}.
\end{remark}

\begin{proof}
Without loss of generality we assume $\mathbb{E}[X]>0$. Then, we have
\begin{align}
\text{Ent}(X) & = \mathbb{E}\left[ X \ln \frac{X}{\mathbb{E}[X]} \right] \\
& = \mathbb{E}[X] \cdot \mathbb{E}\left[ \frac{X}{\mathbb{E}[X]} \ln \frac{X}{\mathbb{E}[X]} \right] \\
\end{align}
Denote the distribution of $X$ as $Q$, and introduce a new probability measure $P$ via the Radon--Nikodym derivative
\begin{align}
\frac{dP}{dQ} & = \frac{X}{\mathbb{E}[X]}, 
\end{align}
we have
\begin{align}
\text{Ent}(X) & = \mathbb{E}_Q[X] \mathbb{E}_Q \left[ \frac{dP}{dQ} \ln \frac{dP}{dQ} \right] \\
& = \mathbb{E}_Q[X] \cdot D(P,Q),
\end{align}
where $D(P,Q)$ is the KL divergence between $P$ and $Q$. 

Applying Lemma~\ref{lemma.tsybakovinequalities} in Appendix~\ref{sec.auxlemmas}, we have
\begin{align}
\text{Ent}(X) & \leq \mathbb{E}_Q[X] \cdot \ln \left( \mathbb{E}_Q \left( \frac{X}{\mathbb{E}_Q[X]} \right)^2 \right) \\
& = \mathbb{E}[X] \cdot \ln \left( 1 + \frac{\mathsf{Var}(X)}{(\mathbb{E}[X])^2} \right),
\end{align}
where in the last step we used the fact that $\mathsf{Var}(X) = \mathbb{E}[X^2] - (\mathbb{E}[X])^2$. Using the fact that $\sup_{x\geq 0} \frac{\ln(1+x)}{\sqrt{x}}< 1$, we have
\begin{align}
\text{Ent}(X) & \leq \sqrt{\mathsf{Var}(X)}. 
\end{align}

Now we prove the lower bounds. Applying the Hellinger distance part of Lemma~\ref{lemma.tsybakovinequalities} in Appendix~\ref{sec.auxlemmas}, we have
\begin{align}
\text{Ent}(X) & \geq \mathbb{E}_Q[X] \cdot H^2(P,Q) \\
& = \mathbb{E}_Q[X] \cdot \mathbb{E}_Q \left( \sqrt{\frac{X}{\mathbb{E}_Q[X]}} -1 \right)^2 \\
& \geq 2 \left( \mathbb{E}_Q[X] - \sqrt{\mathbb{E}_Q[X]} \mathbb{E}_Q[\sqrt{X}] \right) \\
& \geq \mathbb{E}_Q[X] - (\mathbb{E}_Q[\sqrt{X}])^2 \\
& = \mathsf{Var}(\sqrt{X}). 
\end{align}
Here in the last inequality we used the fact that $\mathbb{E}[X] + (\mathbb{E}[\sqrt{X}])^2 - 2 \sqrt{\mathbb{E}[X]} \mathbb{E}[\sqrt{X}]\geq 0$. 

Applying the total variation distance part of Lemma~\ref{lemma.tsybakovinequalities} in Appendix~\ref{sec.auxlemmas}, we have
\begin{align}
\text{Ent}(X) & \geq \mathbb{E}_Q[X] \cdot 2V^2(P,Q) \\
& = \frac{1}{2} \mathbb{E}_Q[X] \cdot \left( \mathbb{E}_Q \left | \frac{X}{\mathbb{E}_Q[X]} -1 \right | \right)^2.
\end{align}

\end{proof}

\section{Auxiliary lemmas}\label{sec.auxlemmas}

\begin{lemma}[Mean value theorem for divided difference]\label{lemma.meanvaluetheoremdivideddifferrence}
Suppose the function $f$ is $n-1$ times differentiable in the interval determined by the smallest and the largest of the $x_i$'s, we have
\begin{align}
f[x_1,x_2,\ldots,x_n] = \frac{f^{(n-1)}(\xi)}{(n-1)!},
\end{align}
where $\xi$ is in the open interval $(\min_i x_i, \max_i x_i)$, and $f[x_1,x_2,\ldots,x_n]$ is the divided difference in Definition~\ref{def.divideddifference}.  
\end{lemma}

The following lemma which is closely related to the mean value theorem for divided differences in the continuous case.
\begin{lemma}\label{lemma.meanvaluetheorem}
For integers $x_0<x_1<\cdots<x_r$ and any function $f$ defined on $\mathbb{Z}$, the following holds:
\begin{align}
|f[x_0,\cdots,x_r]| \le \frac{1}{r!}\max_{x\in [x_0,x_r]} |\Delta^r f(x)|.
\end{align}
Here $\Delta^r f(x)$ denotes the $r$-th order backward difference of $f$, which is defined as
\begin{align}
\Delta^r f(x) \triangleq \sum_{k = 0}^r (-1)^k \binom{r}{k} f(x-k). 
\end{align}
\end{lemma}

\begin{lemma}\label{lemma.2jackkniferepres}
Suppose one observes $X\sim \mathsf{B}(n,p)$. Then, the $r$-jackknife estimator with $n_1 = n-1, n_2 = n, r = 2$ in estimating $f(p)$ in~(\ref{eqn.generalrjackknifedef}) can be represented as
\begin{align}
\hat{f}_2 & = n f\left( \frac{X}{n} \right) - \frac{n-1}{n} \Bigg( (n-X) f\left( \frac{X}{n-1}\right) \nonumber \\ 
& \quad + X f\left( \frac{X-1}{n-1} \right) \Bigg),
\end{align}
where one conveniently sets $f(x) =0$ if $x<0$. 
\end{lemma}

\begin{lemma}\label{lemma.higherpower}
Let $r\geq 2$. Then, for the coefficients given in~(\ref{eqn.ciexpressions}), we have the following. 
\begin{enumerate}
\item If $\rho = 0$, then
\begin{align}
\sum_{i = 1}^r \frac{C_i}{n_i^\rho}  & = 1.
\end{align}
\item If $1\leq \rho \leq r-1$, then
\begin{align}
\sum_{i = 1}^r \frac{C_i}{n_i^\rho}  & = 0. 
\end{align}
\item If $\rho \geq r$, then
\begin{align}
\left | \sum_{i = 1}^r \frac{C_i}{n_i^\rho} \right | & \leq \left| \prod_{s = 0}^{r-2} (r-1-\rho -s) \right | \frac{1}{(r-1)!} \frac{1}{n_1^\rho} \\
& \leq \frac{(\rho-1)^{r-1} }{(r-1)!} \frac{1}{n_1^\rho}. 
\end{align}
\end{enumerate}
\end{lemma}

Define $T_{n,s}(p) = n^s \mathbb{E}_p (\hat{p}_n - p)^s, n = 1,2,\ldots, s = 0,1,\ldots$. We have $T_{n,0} = 1, T_{n,1} = 0$. Upon observing the recurrence relation
\begin{align}
T_{n, s+1}(x) = x(1-x) \left( T'_{n,s}(x) + ns T_{n,s-1}(x) \right),
\end{align}
one obtains the following result. 
\begin{lemma}\cite[Chapter 10, Theorem 1.1.]{devore1993constructive}\label{lemma.binomialmoments}
For a fixed $s = 0,1,\ldots$, $T_{n,s}(p)$ is a polynomial in $p$ of degree $\leq s$, and in $n$ of degree $\lfloor s/2 \rfloor$. Moreover, for $\varphi^2 = p(1-p)$, we have
\begin{align}
T_{n,2s}(p)& = \sum_{j = 1}^s a_{j,s}(\varphi^2) n^j \varphi^{2j} \\
T_{n,2s+1}(p) & = (1-2p) \sum_{j = 1}^s b_{j,s}(\varphi^2) n^j \varphi^{2j},
\end{align}
where $a_{j,s}, b_{j,s}$ are polynomials of degree $\leq s-j$, with coefficients independent of $n$. 
\end{lemma}

\begin{lemma}\label{lemma.centralmomentcoeffbounds}
The central moments of $\hat{p}_n$ where $n\cdot \hat{p}_n \sim \mathsf{B}(n,p)$ satisfy the following: 
\begin{align}
n^s \mathbb{E}_p (\hat{p}_n - p)^s & = \sum_{j = 1}^{\lfloor s/2 \rfloor} h_{j,s}(p) n^j, \label{eqn.individualterm}
\end{align}
where 
\begin{align}
\|h_{j,s}(p)\| \leq \frac{(4es)^s}{j!}.
\end{align}
\end{lemma}

\begin{lemma}[Chernoff bound]\cite{angluin1979fast}\label{lemma.chernoff}
Let $X_1,X_2,\ldots,X_n$ be independent $\{0,1\}$ valued random variables with $\mathbb{P}(X_i = 1) = p_i$. Denote $X = \sum_{i = 1}^n X_i, \mu = \mathbb{E}[X]$. Then,
\begin{align}
\mathbb{P}(X\leq (1-\beta) \mu)  & \leq e^{-\beta^2 \mu /2}\quad 0<\beta \leq 1 \\
\mathbb{P}(X\geq (1+\beta)\mu) & \leq \begin{cases} e^{-\frac{\beta^2 \mu}{2+\beta}} \leq e^{-\frac{\beta\mu}{3}} & \beta >1 \\ e^{-\frac{\beta^2 \mu}{3}} & 0<\beta \leq 1 \end{cases} 
\end{align}
\end{lemma}

\begin{lemma}\cite{strukov1977mathematical}\label{lemma.strukov}
For any continuous function $f: \mathbb{R} \mapsto \mathbb{R}$ and any random variable $X$ taking values in $\mathbb{R}$, we have
\begin{align}
|\mathbb{E}[f(X)] - f(\mathbb{E}[X])| & \leq 3 \cdot \omega^2 \left(f, \frac{\sqrt{\mathsf{Var}(X)}}{2} \right). 
\end{align}
If $f$ is only defined on an interval $[a,b]$ that is a strict subset of $\mathbb{R}$, the result holds with the constant $3$ replaced by $15$. Here $\omega^r(f,t) \triangleq \sup_{0<h\leq t} \| \Delta_h^r f\|$, where $\Delta_h^r f(x) = \sum_{k = 0}^r (-1)^k {r \choose k} f(x + r(h/2) - kh)$, and $\Delta^r_h f(x) = 0$ if $x + rh/2$ or $x-rh/2$ is not inside the domain of $f$.  
\end{lemma}

\begin{lemma}\cite[Section 2.4]{Tsybakov2008}\label{lemma.tsybakovinequalities}
Suppose $P,Q$ are both probability measures, and $P \ll Q$. Introduce the following divergence functionals:
\begin{enumerate}
\item Total variation distance:
\begin{align}
V(P,Q) & = \frac{1}{2} \mathbb{E}_Q \left| \frac{dP}{dQ} -1 \right|;
\end{align}
\item Hellinger distance:
\begin{align}
H(P,Q) & = \left( \mathbb{E}_Q \left( \sqrt{\frac{dP}{dQ}} -1 \right)^2 \right)^{1/2};
\end{align}
\item Kullback--Leibler (KL) divergence:
\begin{align}
D(P,Q) & = \mathbb{E}_Q \left( \frac{dP}{dQ} \ln \frac{dP}{dQ} \right);
\end{align}
\item $\chi^2$ divergence:
\begin{align}
\chi^2(P,Q) & = \mathbb{E}_Q \left( \frac{dP}{dQ} -1 \right)^2 \\
& = \mathbb{E}_Q \left( \frac{dP}{dQ} \right)^2 -1. 
\end{align} 
\end{enumerate}
Then, we have the following upper and lower bounds on the KL divergence:
\begin{align}
D(P,Q) & \leq \ln\left( 1 + \chi^2(P,Q) \right) \\
D(P,Q) & \geq 2 V^2(P,Q) \\
D(P,Q) & \geq H^2(P,Q). 
\end{align}
\end{lemma}

\section{Proofs of main theorems and lemmas} \label{sec.proofmaintheoremslemmas}
\subsection{Proof of Theorem~\ref{thm.jackknife}}

Recognizing $\mathbb{E}_p \hat{f}_r$ as linear combination of operators, the first and second parts of Theorem~\ref{thm.jackknife} follow from~\cite[Theorem 9.3.2.]{Ditzian--Totik1987}, the third part follows from~\cite[Corollary 9.3.8.]{Ditzian--Totik1987}, and the last part follows from~\cite{Totik1994}.

\subsection{Proof of Theorem~\ref{thm.delete1jackknifeworst}}

Define $f\in C(0,1]$ to be the piecewise linear interpolation function at nodes $\left \{ \left( m^{-1},\frac{1+(-1)^{m}}{2} \right) , m\in \mathbb{N}_+ \right \}$. Clearly $f(m^{-1}) = 1$ when $m$ is even, $f(m^{-1}) = 0$ when $m$ is odd, and $\| f\| \leq 1$. We set $f(0) = 0$. 

We have
\begin{align}
& |\EE_p \hat{f}_r - f(p)| \nonumber \\ 
= & \left| \sum_{i=1}^r C_i\EE_p f(\hat{p}_{n+i-r}) - f(p)\right| \\
\ge & \left| \sum_{n+i-r\text{ is even}} C_i\EE_pf(\hat{p}_{n+i-r})\right| \nonumber \\
& - \left| \sum_{n+i-r\text{ is odd}} C_i\EE_pf(\hat{p}_{n+i-r})\right|  - |f(p)|\\
\ge & \sum_{n+i-r\text{ is even}} |C_i|\EE_pf(\hat{p}_{n+i-r}) \nonumber \\
& - \sum_{n+i-r\text{ is odd}} |C_i|\EE_pf(\hat{p}_{n+i-r})  - 1
\end{align}
where in the last step we have used the fact that $C_1,\cdots,C_r$ have alternating signs, and $f\geq 0$. 

When $n+i-r$ is even and $p = n^{-1}$, 
\begin{align}
\EE_p f(\hat{p}_{n+i-r}) &\ge f \left (\frac{1}{n+i-r} \right )\cdot \mathbb{P}(\mathsf{B}(n+i-r,p)=1) \\
&= (n+i-r)p\left(1-p\right)^{n+i-r-1}\\
&\ge \frac{1}{e}(1-o(1)).
\end{align}

When $n+i-r$ is odd and $p=n^{-1}$, noting that $f(0) = f\left( \frac{1}{n+i-r} \right) = 0$, we have
\begin{align}
\EE_p f(\hat{p}_{n+i-r}) &\le \|f\|\cdot \mathbb{P}(\mathsf{B}(n+i-r,p)\geq 2) \\
& \leq \left(1-\frac{2}{e}\right)(1+o(1)).
\end{align}

Since $\sum_{i=1}^r C_i=1$, we have $\sum_{n+i-r\text{ is odd}}|C_i|=(1+o(1))\sum_{n+i-r\text{ is even}}|C_i| \asymp n^{r-1}$.  Combining these together, we arrive at
\begin{align}
|\EE_p \hat{f}_r - f(p)| &\gtrsim \sum_{n+i-r\text{ is even}}|C_i| \left( \frac{1}{e} - \left( 1- \frac{2}{e} \right) - o(1) \right) \\
& \gtrsim n^{r-1}. 
\end{align}
which completes the proof of the first claim. 

As for the second claim, it suffices to replace the function $f$ on interval $[0,1/n]$ by the linear interpolation function interpolating $f(0) = 0$ and $f(1/n) = \frac{1 + (-1)^n}{2}$ and keep other parts of the function intact. Consequently, after this modification $f \in C[0,1]$. 

Now we prove the variance part. 

Construct $f\in C[0,1]$ to be a piecewise linear interpolation function at the following nodes: $f(0) = 0, f\left( \frac{1}{n} \right) = f\left( \frac{2}{n-1} \right) = 1, f\left( \frac{1}{n-1}\right) = f\left(\frac{2}{n}\right) = -1, f(1) = 0$.  

It follows from straightforward algebra and Lemma~\ref{lemma.2jackkniferepres} that
\begin{align}
\hat{f}_2 & = \begin{cases} 0 & X = 0\\ 2n-2 + n^{-1} & X = 1 \\ -2n + 5 -\frac{4}{n} & X = 2 \end{cases}. 
\end{align}

It follows from the definition of variance that
\begin{align}
\mathsf{Var}_p(\hat{f}_2) & = \mathbb{E}_p \left( \hat{f}_2 - \mathbb{E}_p \hat{f}_2 \right)^2 \\
& = \inf_a \mathbb{E}_p \left( \hat{f}_2 -a \right)^2 \\
& \geq \inf_a \big( \mathbb{P}(\mathsf{B}(n,p) = 1) ( 2n -2 + n^{-1} -a)^2 \nonumber \\ 
& \quad + \mathbb{P}(\mathsf{B}(n,p) = 2) (-2n+5-4/n-a)^2 \big)  \\
& = \inf_a \big( np(1-p)^{n-1}( 2n -2 + n^{-1} -a)^2 \nonumber \\ 
& \quad + \frac{n(n-1)}{2}p^2 (1-p)^{n-2} (-2n+5-4/n-a)^2 \big)
\end{align}
Setting $p = 1/n$, we have
\begin{align}
\mathsf{Var}_p(\hat{f}_2) & \geq \frac{1}{2}\left( 1-\frac{1}{n}\right)^{n-1} \inf_a \left( ( 2n -2 + n^{-1} -a)^2 \right. \nonumber \\ 
& \quad \left. + (-2n+5-4/n-a)^2 \right).
\end{align}
The infimum is achieved when 
\begin{align}
a & = \frac{(2n -2 + n^{-1}) + (-2n+5-4/n)}{2} \\
& = \frac{3}{2} \left( 1- \frac{1}{n} \right).
\end{align}
Hence,
\begin{align}
\| \mathsf{Var}_p(\hat{f}_2) \| & \geq  \frac{1}{2}\left( 1-\frac{1}{n}\right)^{n-1} \cdot 2 \nonumber \\ 
& \quad \cdot \left( 2n -2 + \frac{1}{n} - \frac{3}{2} \left( 1-\frac{1}{n} \right) \right)^2 \\
& \geq \left( 1-\frac{1}{n}\right)^{n-1} n^2 \\
& \geq \frac{n^2}{e}. 
\end{align}
where we have used $n\geq 4$ and $\left( 1- \frac{1}{n} \right)^{n-1} \geq e^{-1}$.

\subsection{Proof of Lemma~\ref{lemma.Taylorjackknife}}

Since $f$ satisfies condition $D_s$, it admits the Taylor expansion:
\begin{align}
f(x) & = f(p) + \sum_{u = 1}^{s-1} \frac{f^{(u)}(p)}{u!} (x-p)^u + R_{s}(x;p) 
\end{align}
Applying the $r$-jackknife estimator on it, we have
\begin{align}
\mathbb{E}_p \hat{f}_r & = f(p) + \sum_{u = 1}^{s-1} \frac{f^{(u)}(p)}{u!} \sum_{i = 1}^r C_i \mathbb{E}_p (\hat{p}_{n_i} - p)^u   \nonumber \\ 
&\quad + \sum_{i = 1}^r C_i \mathbb{E}_p R_{s}(\hat{p}_{n_i};p) \\
& = f(p) + \sum_{u = r+1}^{s-1} \frac{f^{(u)}(p)}{u!} \sum_{i = 1}^r C_i \mathbb{E}_p (\hat{p}_{n_i} - p)^u\nonumber \\ 
& \quad + \sum_{i = 1}^r C_i \mathbb{E}_p R_{s}(\hat{p}_{n_i};p), \label{eqn.entireexpansion}
\end{align}
where in the last step we have used Lemma~\ref{lemma.higherpower} and Lemma~\ref{lemma.binomialmoments}. By convention $\sum_{a}^b = 0$ if $a>b$. 

Denote $E_u = \sum_{i = 1}^r C_i \mathbb{E}_p (\hat{p}_{n_i} - p)^u, u\geq r+1$, it follows from Lemma~\ref{lemma.centralmomentcoeffbounds} that 
\begin{align}
E_u & = \sum_{i = 1}^r C_i \sum_{j = 1}^{\lfloor u/2 \rfloor} h_{j,u}(p) \frac{1}{n_i^{u-j}} = \sum_{j = 1}^{\lfloor u/2 \rfloor} h_{j,u}(p) \sum_{i = 1}^r \frac{C_i}{n_i^{u-j}}.
\end{align}
Note that 
\begin{align}
\sum_{i = 1}^r \frac{C_i}{n_i^{u-j}} \neq 0
\end{align}
if and only if $u-j\geq r$, and when that is the case, we have
\begin{align}
\left | \sum_{i = 1}^r \frac{C_i}{n^{u-j}} \right | & \leq (u-j-1)^{r-1} \frac{1}{n_1^{u-j}}
 \leq \frac{u^{r-1}}{n_1^{u-j}} \leq \frac{u^{r-1}}{n_1^r} \label{eqn.boundeachterm}
\end{align}

Since $r$ is fixed, it follows from (\ref{eqn.boundeachterm}), Lemma~\ref{lemma.centralmomentcoeffbounds} and condition $D_s$ with parameter $L$ that 
\begin{align}
\left | \sum_{u = r+1}^{s-1} \frac{f^{(u)}(p)}{u!} E_u \right | \lesssim_{r,s,L} \frac{1}{n^r}. 
\end{align}

It follows from the integral form of Taylor remainder that
\begin{align}
R_{s}(x;p) & = \frac{1}{(s-1)!} \int_p^x (x-t)^{s-1} f^{(s)}(t) dt
\end{align}
Thus we know that when  $s$ is even, we have
\begin{align}
    R_{s}(x;p) 
& =  \frac{1}{(s-1)!} \int_0^1 |x-t|^{s-1} f^{(s)}(t) \nonumber \\ 
& \quad \cdot \mathbbm{1}_{t \in [\min\{x,p\}, \max\{x,p\}]} dt.
\end{align}
When $s$ is odd, we have
\begin{align}
    R_{s}(x;p) 
& =  \frac{1}{(s-1)!} \int_0^1 |x-t|^{s-1} f^{(s)}(t)\nonumber \\ 
& \quad \cdot \mathbbm{1}_{t \in [\min\{x,p\}, \max\{x,p\}]} (-1)^{\mathbbm{1}_{x<p}} dt. 
\end{align}
When $s$ is even, we further have
\begin{align}
& \mathbb{E}_p |\hat{p}_{n_i} -t|^{s-1} \mathbbm{1}_{t \in [\min\{\hat{p}_{n_i},p\}, \max\{\hat{p}_{n_i},p\}]} \nonumber \\ 
=&  \begin{cases} \mathbb{E}_p (\hat{p}_{n_i}-t)_+^{s-1} & t\geq p \\ \mathbb{E}_p (\hat{p}_{n_i}-t)_{-}^{s-1} & t< p    \end{cases}
\end{align}
where $(x)_+ = \max\{x,0\}, (x)_{-} = \max\{-x,0\}$. 

When $s$ is odd, 
\begin{align}
& \mathbb{E}_p |\hat{p}_{n_i} -t|^{s-1} \mathbbm{1}_{t \in [\min\{\hat{p}_{n_i},p\}, \max\{\hat{p}_{n_i},p\}]} (-1)^{\mathbbm{1}_{\hat{p}_{n_i}<p}} \nonumber \\ 
= &  \begin{cases} \mathbb{E}_p (\hat{p}_{n_i}-t)_+^{s-1} & t\geq p \\ -\mathbb{E}_p (\hat{p}_{n_i}-t)_{-}^{s-1} & t< p    \end{cases}
\end{align}

Hence, 
\begin{align}
& \sum_{i = 1}^r C_i \mathbb{E}_p R_{s}(\hat{p}_{n_i};p) \nonumber \\ 
 = & \frac{1}{(s-1)!} \Bigg ( \int_0^p f^{(s)}(t) \left( \sum_{i = 1}^r C_i \mathbb{E}_p (\hat{p}_{n_i}-t)_{-}^{s-1} \right.\nonumber \\ 
 & \cdot (-1)^s \Bigg) dt \nonumber \\
& + \int_p^1 f^{(s)}(t) \left( \sum_{i =1}^r C_i \mathbb{E}_p (\hat{p}_{n_i} -t)^{s-1}_{+} \right) dt\Bigg ) \\
 \lesssim_{r,s,L} & \int_0^p \left|\sum_{i=1}^r C_i\EE_p(\hat{p}_{n_i}-t)_-^{s-1}\right|dt \nonumber \\ & + \int_p^1 \left|\sum_{i=1}^r C_i\EE_p(\hat{p}_{n_i}-t)_+^{s-1}\right|dt,
\end{align}
where we have used the assumption that $\|f^{(s)}\| \leq L$.

\subsection{Proof of Lemma~\ref{lemma.A_nr}}

First we recall the following additive Chernoff bound: for $n\hat{p}_n\sim \mathsf{B}(n,p)$ and $t\ge p$, we have
\begin{align}
\mathbb{P}(\hat{p}_n>t) \le \exp(-nD(t\|p))
\end{align}
where $D(t\|p)$ denotes the KL divergence between the binary distributions $(t,1-t)$ and $(p,1-p)$. Note that
\begin{align}
D(t\|p) &\ge \frac{1}{2}\min_{\xi\in [p,t]} \frac{d^2D(u\|p)}{du^2}\bigg|_{u=\xi}\cdot (t-p)^2\\
&= \min_{\xi\in[p,t]}\frac{(t-p)^2}{2\xi(1-\xi)}\\
&\ge \min_{\xi\in[p,t]}\frac{(t-p)^2}{2\xi}\\
&= \frac{(t-p)^2}{2t}
\end{align}
we arrive at the following inequality:
\begin{align}\label{eq.chernoff}
\mathbb{P}(\hat{p}_n>t) \le \exp(-nD(t\|p)) \le \exp(-\frac{n(t-p)^2}{2t}).
\end{align}

Now we prove the lemma. Let $X_1,\cdots,X_n$ be i.i.d $\mathsf{Bern}(p)$ random variables, and consider the following coupling between $\hat{p}_{n-s},\cdots,\hat{p}_n$: for $k=0,\cdots,s$, define
\begin{align}
\hat{p}_{n-k} = \frac{1}{n-k}\sum_{i=1}^{n-k}X_i.
\end{align}
In other words, we have
\begin{align}
\hat{p}_{n-k} = \hat{p}_{n-s} + \frac{1}{n-k}\sum_{i=k}^{s-1} (X_{n-i}-\hat{p}_{n-s}).
\end{align}

Now define $g_{u,t}(x)=(x-t)_+^{u}$, by Taylor expansion we have
\begin{align}
& \Delta^s A_{n,u}(t) \nonumber \\
& \quad = \EE_p \sum_{k=0}^s (-1)^k\binom{s}{k}g_{u,t}(\hat{p}_{n-k})\\
&= \EE_p \sum_{k=0}^s (-1)^k\binom{s}{k} \left(\sum_{j=0}^{u-1} \frac{g_{u,t}^{(j)}(\hat{p}_{n-s})}{j!}(\hat{p}_{n-k}-\hat{p}_{n-s})^j\right.\nonumber\\
&\left.\qquad\qquad +\frac{g_{u,t}^{(u)}(\xi_k)}{u!}(\hat{p}_{n-k}-\hat{p}_{n-s})^{u}\right)\\
&= \EE_p \sum_{k=0}^s (-1)^k\binom{s}{k} \left(\sum_{j=0}^{u} \frac{g_{u,t}^{(j)}(\hat{p}_{n-s})}{j!}(\hat{p}_{n-k}-\hat{p}_{n-s})^j\right.\nonumber\\
&\left.\quad +\frac{g_{u,t}^{(u)}(\xi_k)-g_{u,t}^{(u)}(\hat{p}_{n-s})}{u!}(\hat{p}_{n-k}-\hat{p}_{n-s})^{u}\right)\\
&= \EE_p \sum_{j=0}^{u}\frac{g_{u,t}^{(j)}(\hat{p}_{n-s})}{j!}\sum_{k=0}^s (-1)^k\binom{s}{k} \nonumber \\ 
& \quad \cdot \EE_p[(\frac{1}{n-k}\sum_{i=k}^{s-1}(X_{n-i}-\hat{p}_{n-s}))^j|X^{n-s}] \nonumber\\
&\quad + \EE_p \sum_{k=0}^s (-1)^k\binom{s}{k}\frac{g_{u,t}^{(u)}(\xi_k)-g_{u,t}^{(u)}(\hat{p}_{n-s})}{u!}\nonumber \\ 
& \quad \cdot (\hat{p}_{n-k}-\hat{p}_{n-s})^{u}\\
&\equiv \EE_p\sum_{j=0}^{u}\frac{g_{u,t}^{(j)}(\hat{p}_{n-s})}{j!}A_j + \EE_pB_{u}.\label{eq.Taylor_exp}
\end{align}
Note that $g_{u,t}(x)$ is in fact not $u$-times differentiable at $x=t$, but with the convention that $g_{u,t}^{(u)}(t)$ can stand for any number in $[0,u!]$, the previous formula remains valid.

\subsubsection{Non-remainder term $A_j$}
Further define
\begin{align}
a_j(t) \triangleq \EE_p (X_n-t)^j = p(1-t)^j + (1-p)(-t)^j
\end{align}
we have $a_0(t)=1, a_1(t)=p-t$, and
\begin{align}
A_j &= \sum_{k=0}^{s-1} (-1)^k\binom{s}{k} \frac{1}{(n-k)^j} \nonumber \\ 
& \quad \cdot \sum_{i_1+\cdots+i_{s-k}=j}\binom{j}{i_1\ \cdots \ i_{s-k}}\prod_{l=1}^{s-k} a_{i_l}(\hat{p}_{n-s}).
\end{align}
Denote by ${I}_j$ the set of all multi-indices $i=(i_1,\cdots,i_{t(i)})$ with $t(i)\le s, \sum_{l=1}^{t(i)} i_l=j$ and $i_l\ge 1$, then for any $i\in\mathcal{I}_j$, the coefficient of $\prod_{l=1}^{t(i)} a_{i_l}(\hat{p}_{n-s})$ in $A_j$ is
\begin{align}
b_i &= \binom{j}{i_1\ \cdots \ i_{t(i)}}\sum_{k=0}^{s-1}(-1)^k\binom{s}{k}\frac{1}{(n-k)^j} \nonumber \\ 
& \quad \cdot \binom{s-k}{d_1\ \cdots\ d_{t'}\ s-k-\sum_{l'=1}^{t'} d_{l'}}\\
&= \binom{j}{i_1\ \cdots \ i_{t(i)}}\binom{t(i)}{d_1\ \cdots \ d_{t'}}\sum_{k=0}^{s-1}(-1)^k\binom{s}{k}\nonumber \\ 
& \quad \cdot \frac{1}{(n-k)^j}\binom{s-k}{t(i)}\\
&=  \binom{j}{i_1\ \cdots \ i_{t(i)}}\binom{t(i)}{d_1\ \cdots \ d_{t'}}\binom{s}{t(i)}\sum_{k=0}^{s-t(i)}(-1)^k\nonumber \\ 
& \quad \cdot\binom{s-t(i)}{k}\frac{1}{(n-k)^j}\\
&=  \binom{j}{i_1\ \cdots \ i_{t(i)}}\binom{t(i)}{d_1\ \cdots \ d_{t'}}\binom{s}{t(i)}\cdot \Delta^{s-t(i)}n^{-j}
\end{align}
where $d_1,\cdots,d_{t'}$ is the nonzero histograms of $i=(i_1,\cdots,i_{t(i)})$ with $\sum_{l'=1}^{t'} d_{l'}=t(i)$. By the mean value theorem of backward differences, we have
\begin{align}
|\Delta^{s-t(i)}n^{-j}| & \le \max_{x\in [n-s+t(i),n]} \left|\frac{d^{s-t(i)}}{dx^{s-t(i)}}(x^{-j})\right| \nonumber \\ 
&\lesssim_s n^{-(j+s-t(i))}
\end{align}
and thus
\begin{align}
|b_i| \le c(i)n^{-(j+s-t(i))}
\end{align}
where the constant $c(i)$ does not depend on $n$ or $p$.

Now for any $j=0,1,\cdots,u$, we have
\begin{align}
&\left|\EE_p\frac{g_{u,t}^{(j)}(\hat{p}_{n-s})}{j!}A_j\right| \nonumber \\ 
\asymp &\left|\EE_p\left[(\hat{p}_{n-s}-t)_+^{u-j}\cdot \sum_{i\in I_j}b_i\prod_{l=1}^{t(i)}a_{i_l}(\hat{p}_{n-s})\right]\right|\\
\le & \sum_{i\in I_j} c(i)n^{-(j+s-t(i))}\nonumber \\ 
& \quad \cdot \EE_p\left[(\hat{p}_{n-s}-t)_+^{u-j}\prod_{l=1}^{t(i)}|a_{i_l}(\hat{p}_{n-s})|\right].
\end{align}
For any $i\in I_j$, denote by $v(i)$ the number of ones in $i=(i_1,\cdots,i_t)$, using $a_1(t)=p-t$ and $|a_j(t)|\le p+t$ yields
\begin{align}
& \EE_p\left[(\hat{p}_{n-s}-t)_+^{u-j}\prod_{l=1}^{t(i)}|a_{i_l}(\hat{p}_{n-s})|\right] \nonumber \\ 
\le & \EE_p\left[(\hat{p}_{n-s}-t)_+^{u-j} |\hat{p}_{n-s}-p|^{v(i)}|\hat{p}_{n-s}+p|^{t(i)-v(i)}\right].
\end{align}

We distinguish into two cases. If $np\ge 1$, by Holder's inequality $\EE[XYZ]^3\le \EE[|X|^3]\EE[|Y|^3]\EE[|Z|^3]$ we obtain
\begin{align}
& \EE_p\left[(\hat{p}_{n-s}-t)_+^{u-j}\prod_{l=1}^{t(i)}|a_{i_l}(\hat{p}_{n-s})|\right] \nonumber \\
& \le \EE_p\left[(\hat{p}_{n-s}-t)_+^{u-j} |\hat{p}_{n-s}-p|^{v(i)}|\hat{p}_{n-s}+p|^{t(i)-v(i)}\right]\\
&\le \EE_p\Big[|\hat{p}_{n-s}-p|^{u-j+v(i)} \mathbbm{1}(\hat{p}_{n-s}\ge t) \nonumber \\ 
& \quad \cdot|\hat{p}_{n-s}+p|^{t(i)-v(i)}\Big]\\
&\le \left(\EE_p|\hat{p}_{n-s}-p|^{3(u-j+v(i))}\right)^{\frac{1}{3}} \nonumber \\ 
& \quad \cdot\left(\EE_p|\hat{p}_{n-s}+p|^{3(t(i)-v(i))}\right)^{\frac{1}{3}}\cdot \sqrt[3]{\mathbb{P}(\hat{p}_{n-s}>t)}\\
&\lesssim_{u,s} (\frac{p}{n})^{\frac{u-j+v(i)}{2}}\cdot p^{t(i)-v(i)}\cdot \exp(-\frac{c_2n(t-p)^2}{t})\\
&= n^{-\frac{u-j+v(i)}{2}}p^{\frac{u-j+2t(i)-v(i)}{2}}\cdot \exp(-\frac{c_2n(t-p)^2}{t})
\end{align}
where in the last step we have used \eqref{eq.chernoff}, and the fact that when $np\ge 1$, we have~\cite{Han--Jiao--Weissman2016minimax}
\begin{align}
\EE_p|\hat{p}_n-p|^k &\lesssim_k (\frac{p}{n})^{\frac{k}{2}}\\
\EE_p|\hat{p}_n+p|^k &\lesssim_k p^k.
\end{align}
As a result, in this case we conclude that
\begin{align}
& \left|\EE_p\frac{g_{u,t}^{(j)}(\hat{p}_{n-s})}{j!}A_j\right| \nonumber \\
\lesssim_{u,s} &\sum_{i\in I_j} c(i)n^{-(j+s-t(i))}\cdot n^{-\frac{u-j+v(i)}{2}} \nonumber \\ 
& \cdot p^{\frac{u-j+2t(i)-v(i)}{2}} \exp(-\frac{c_2n(t-p)^2}{t})\\
=\quad & \sum_{i\in I_j} c(i)n^{-(\frac{u}{2}+s)}p^{\frac{u}{2}}\cdot (np)^{-\frac{j+v(i)-2t(i)}{2}}\nonumber \\ 
& \cdot \exp(-\frac{c_2n(t-p)^2}{t}).
\end{align}
Note that $j=\sum_{k=1}^{t(i)}i_k\ge v(i)+2(t(i)-v(i))$, we have $j+v(i)\ge 2t(i)$, and thus by $np\ge 1$, 
\begin{align}
& \left|\EE_p\frac{g_{u,t}^{(j)}(\hat{p}_{n-s})}{j!}A_j\right| \nonumber \\
&\quad \lesssim_{u,s} \sum_{i\in I_j} c(i)n^{-(\frac{u}{2}+s)}p^{\frac{u}{2}}\cdot \exp(-\frac{c_2n(t-p)^2}{t})\\
&\quad \lesssim_{u,s} n^{-(\frac{u}{2}+s)}p^{\frac{u}{2}}\cdot \exp(-\frac{c_2n(t-p)^2}{t})
\end{align} 
where we have used that both $|I_j|$ and $c(i)$ do not depend on $n$ or $p$ in the last step.

If $np<1$, we first show that
\begin{align}\label{eq.inequality_small_p}
\EE_p [\hat{p}_n^k\mathbbm{1}(\hat{p}_n>t)] \lesssim_k \frac{p}{n^{k-1}}\exp(-\frac{c_2n(t-p)^2}{t}).
\end{align}
In fact, the MGF of $\hat{p}_n$ gives
\begin{align}
\EE_p[e^{\lambda\hat{p}_n}] = (pe^{\frac{\lambda}{n}}+1-p)^n.
\end{align}
Differentiating w.r.t $\lambda$ for $k$ times, for $\lambda>0$ we arrive at

\begin{align}
& \EE_p[\hat{p}_n^k e^{\lambda\hat{p}_n}] \nonumber \\
&  = \frac{d^k}{d\lambda^k}\left[(pe^{\frac{\lambda}{n}}+1-p)^n\right] \\ &\le C_k\sum_{j=1}^k (pe^{\frac{\lambda}{n}})^jn^{j-k}(pe^{\frac{\lambda}{n}}+1-p)^{n-j} \\
&\lesssim_k \frac{pe^{\frac{\lambda}{n}}(pe^{\frac{\lambda}{n}}+1-p)^{n-1}}{n^{k-1}} \nonumber \\ & \quad \cdot \left(1+\left(\frac{npe^{\frac{\lambda}{n}}}{pe^{\frac{\lambda}{n}}+1-p}\right)^{k-1}\right)\\
&\lesssim_k \frac{pe^{\frac{k\lambda}{n}}(pe^{\frac{\lambda}{n}}+1-p)^{n}}{n^{k-1}},
\end{align}
where $C_k$ is a universal constant depending on $k$ only, and in the last step we have used $np<1$. As a result, by Markov's inequality we have for any $\lambda>0$, 
\begin{align}
\EE_p[\hat{p}_n^k\mathbbm{1}(\hat{p}_n>t)] &\le \EE_p[\hat{p}_n^ke^{\lambda(\hat{p}_n-t)}] \nonumber \\ 
& \lesssim_k e^{-\lambda t}\cdot \frac{pe^{\frac{k\lambda}{n}}(pe^{\frac{\lambda}{n}}+1-p)^{n}}{n^{k-1}}.
\end{align}
Specifically, when $t> \frac{2k}{n}$, choosing $\lambda=n\ln \frac{(1-p)t}{(2-t)p}$ yields
\begin{align}
\EE_p[\hat{p}_n^k\mathbbm{1}(\hat{p}_n>t)] & \lesssim_k \frac{p}{n^{k-1}}\left(\frac{(2-t)p}{(1-p)t}\right)^{nt-k}\nonumber \\ & \quad \cdot \left(\frac{1-p}{1-\frac{t}{2}}\right)^n\\
&\le \frac{p}{n^{k-1}}\left(\frac{(2-t)p}{(1-p)t}\right)^{\frac{nt}{2}}\cdot \left(\frac{1-p}{1-\frac{t}{2}}\right)^n\\
&\le \frac{p}{n^{k-1}}\exp(-nD(\frac{t}{2}\|p))\\
&\le \frac{p}{n^{k-1}}\exp(-\frac{c_2n(t-p)^2}{t})
\end{align}
as desired. When $p<t\le \frac{2k}{n}$, we have $\frac{c_2n(t-p)^2}{t}=O(1)$, and \eqref{eq.inequality_small_p} follows from 
\begin{align}
\EE_p [\hat{p}_n^k\mathbbm{1}(\hat{p}_n < t)] \le \EE_p \hat{p}_n^k \lesssim_k \frac{p}{n^{k-1}}.
\end{align}

Now based on \eqref{eq.inequality_small_p}, we have
\begin{align}
& \EE_p\left[(\hat{p}_{n-s}-t)_+^{u-j}\prod_{l=1}^{t(i)}|a_{i_l}(\hat{p}_{n-s})|\right] \nonumber \\
&\quad \le \EE_p\Big[(\hat{p}_{n-s}-t)_+^{u-j} |\hat{p}_{n-s}-p|^{v(i)}\nonumber \\ 
& \qquad \cdot|\hat{p}_{n-s}+p|^{t(i)-v(i)}\Big]\\
&\quad \le \EE_p\Big[|\hat{p}_{n-s}-p|^{u-j+v(i)} \mathbbm{1}(\hat{p}_{n-s}\ge t) \nonumber \\ 
& \qquad \cdot|\hat{p}_{n-s}+p|^{t(i)-v(i)}\Big]\\
&\quad \lesssim_{u,s} \EE_p\left[(\hat{p}_{n-s}^{u-j+t(i)}+p^{u-j+t(i)})\mathbbm{1}(\hat{p}_n>t)\right]\\
&\quad \lesssim_{u,s} \left(\frac{p}{n^{u-j+t(i)-1}}+p^{u-j+t(i)}\right)\nonumber \\ 
& \qquad \cdot \exp(-\frac{c_2n(t-p)^2}{t})\\
&\quad \lesssim_{u,s} \frac{p}{n^{u-j+t(i)-1}}\exp(-\frac{c_2n(t-p)^2}{t}).
\end{align}
As a result, in this case we have
\begin{align}
& \left|\EE_p\frac{g_{u,t}^{(j)}(\hat{p}_{n-s})}{j!}A_j\right| \nonumber \\
&\quad \lesssim_{u,s} \sum_{i\in I_j} c(i)n^{-(j+s-t(i))}\cdot \frac{p}{n^{u-j+t(i)-1}}\nonumber \\ 
& \qquad \cdot \exp(-\frac{c_2n(t-p)^2}{t})\\
&\quad = \sum_{i\in I_j} c(i)n^{-(u+s-1)}p\cdot\exp(-\frac{c_2n(t-p)^2}{t})\\
&\quad \lesssim_{u,s} n^{-(u+s-1)}p\cdot\exp(-\frac{c_2n(t-p)^2}{t}).
\end{align}
Moreover, when $\frac{1}{2}<p\le 1-\frac{1}{n}$, we can use the symmetry $n(1-\hat{p}_n)\sim \mathsf{B}(n,1-p)$ and $|a_j(t)|\le (1-p)+(1-t)$, and then adapt the proof of the first case to conclude that
\begin{align}
& \left|\EE_p\frac{g_{u,t}^{(j)}(\hat{p}_{n-s})}{j!}A_j\right|  \nonumber \\ 
\lesssim_{u,s}  & n^{-(\frac{u}{2}+s)}(1-p)^{\frac{u}{2}}\cdot \exp(-\frac{c_2n(t-p)^2}{1-p}).
\end{align}

In summary, for the non-remainder terms, if $p<\frac{1}{n} $,
\begin{align}\label{eq.non-remainder}
    \left|\EE_p\frac{g_{u,t}^{(j)}(\hat{p}_{n-s})}{j!}A_j\right|  
\lesssim_{u,s} & 
n^{-(u+s-1)}p\cdot \exp(-\frac{c_2n(t-p)^2}{t});
\end{align}
 if $\frac{1}{n}\le p\le\frac{1}{2} $,
\begin{align}
    \left|\EE_p\frac{g_{u,t}^{(j)}(\hat{p}_{n-s})}{j!}A_j\right|  
\lesssim_{u,s} & n^{-(\frac{u}{2}+s)}p^{\frac{u}{2}}\nonumber \\ 
& \cdot  \exp(-\frac{c_2n(t-p)^2}{t});
\end{align}
 if $\frac{1}{2}< p\le 1-\frac{1}{n}$,
\begin{align}
    \left|\EE_p\frac{g_{u,t}^{(j)}(\hat{p}_{n-s})}{j!}A_j\right|  
\lesssim_{u,s} &n^{-(\frac{u}{2}+s)}(1-p)^{\frac{u}{2}}\nonumber \\ 
& \cdot \exp(-\frac{c_2n(t-p)^2}{1-p});
\end{align}

\subsubsection{The remainder term $B_{u}$}
By our convention on $g_{u,t}^{(u)}(\cdot)$ we observe that $g_{u,t}^{(u)}(\xi_k)-g_{u,t}^{(u)}(\hat{p}_{n-s})$ is non-zero only if
\begin{align}
(\xi_k - t)(\hat{p}_{n-s}-t) \le 0.
\end{align}
However, since $\min\{\hat{p}_{n-k},\hat{p}_{n-s}\} \le\xi_k\le \max\{\hat{p}_{n-k},\hat{p}_{n-s}\}$, the previous inequality implies that the ``path" consisting of $\hat{p}_{n-s},\cdots,\hat{p}_n$ under our coupling ``walks across" $t$. Let's call it ``good path". Moreover, note that
\begin{align}
|\hat{p}_{n-k} - \hat{p}_{n-s}| \le \frac{s}{n-s}, \qquad k=0,\cdots,s
\end{align}
under our coupling, for a good path we must have
\begin{align}
|\hat{p}_{n-s} - t| \le \frac{s}{n-s}.
\end{align}
Since $\hat{p}_{n-s}$ must be an integral multiple of $\frac{1}{n-s}$, we conclude that the number of good paths is $O(1)$. Let's call $\hat{p}_{n-k}\to \hat{p}_{n-k+1}$ a ``right step" if $\hat{p}_{n-k}<\hat{p}_{n-k+1}$ or $\hat{p}_{n-k}=\hat{p}_{n-k+1}=1$, and a ``left step" if $\hat{p}_{n-k}>\hat{p}_{n-k+1}$ or $\hat{p}_{n-k}=\hat{p}_{n-k+1}=0$. 

First we consider the case where $p\le \frac{1}{2}$, and consider any good path $L$. The idea is that, each good path $L$ gives rise to a realization of $B_u$, and $\EE_pB_u$ is the expectation averages over all good paths. Hence, to evaluate $\EE_p B_u$, it suffices to compute the value of $B_u$ given $L$, and the probability of the path $L$. Denote by $r,l$ the number of right and left steps in $L$, and by $q$ the starting point of $L$, it is easy to see that the probability of this path is
\begin{align}
\mathbb{P}(L) = p^r(1-p)^l\cdot \mathbb{P}(\mathsf{B}(n-s,p)=(n-s)q).
\end{align}

We first take a look at the quantity $\mathbb{P}(\mathsf{B}(n,p)=nq)$ for any $q$ with $|q-t|=O(n^{-1})$. By Stirling's approximation
\begin{align}
n! = \sqrt{2\pi n}\left(\frac{n}{e}\right)^n(1+o(1)),
\end{align}
we have
\begin{align}
&\mathbb{P}(\mathsf{B}(n,p)=nq) \nonumber \\ 
= &  \binom{n}{nq}p^{nq}(1-p)^{n-nq}\\
\lesssim &\frac{\sqrt{2\pi n}(\frac{n}{e})^n}{\sqrt{2\pi nq}(\frac{nq}{e})^{nq}\cdot \sqrt{2\pi (n-nq)}(\frac{n-nq}{e})^{n-nq}}\nonumber \\ 
& \cdot p^{nq}(1-p)^{n-nq}\\
\lesssim& \frac{1}{\sqrt{nq(1-q)}}\cdot e^{-nD(q\|p)}\\
\le &\frac{1}{\sqrt{nq(1-q)}}\cdot \exp(-\frac{n(q-p)_+^2}{2q})
\end{align}
where the last step is given by \eqref{eq.chernoff}. Moreover, for $q>1-\frac{1}{n}$, \eqref{eq.chernoff} gives a better bound
\begin{align}
\mathbb{P}(\mathsf{B}(n,p)=nq) \le \exp(-\frac{n(q-p)^2}{2q}).
\end{align}
Combining them together, we conclude that
\begin{align}
\mathbb{P}(\mathsf{B}(n,p)=nq) \lesssim \frac{1}{\sqrt{nq}}\exp(-\frac{n(q-p)_+^2}{4q})
\end{align}

Now we show that if $q=t+O(n^{-1})$, we can replace $q$ by $t$ without loss in the previous inequality. In fact, if $nq\ge 1$, it is easy to verify that
\begin{align}
\mathbb{P}(\mathsf{B}(n,p)=nq) & \lesssim \frac{1}{\sqrt{nq}}\exp(-\frac{n(q-p)_+^2}{4q}) \\
& \lesssim (\frac{1}{\sqrt{nt}}\wedge 1)\exp(-\frac{n(t-p)^2}{4t})
\end{align}
and if $nq<1$, we use the trivial bound
\begin{align}
\mathbb{P}(\mathsf{B}(n,p)=nq) \le 1 \lesssim (\frac{1}{\sqrt{nt}}\wedge 1)\exp(-\frac{n(t-p)^2}{4t}).
\end{align}
As a result, for all good paths with starting point $q$ we conclude that
\begin{align}
\mathbb{P}(\mathsf{B}(n,p)=nq) &\lesssim (\frac{1}{\sqrt{nt}}\wedge 1)\exp(-\frac{n(t-p)^2}{4t})
\end{align}

Now we evaluate the quantity $B_u(L)$ given $L$. In fact, it is easy to see
\begin{align}
|B_u(L)| \lesssim_u \begin{cases}
n^{-u} & \text{if }r>0\\
(\frac{p}{n})^u & \text{if }r=0.
\end{cases}
\end{align}
As a result,
\begin{align}
& |B_u(L)| \cdot \mathbb{P}(L) \nonumber \\
&\quad \lesssim_u \left(n^{-u}\cdot p + (\frac{p}{n})^u\right)\cdot \mathbb{P}(\mathsf{B}(n,p)=nq)\\
& \quad \lesssim_u n^{-u}p^{u\wedge 1}\cdot (\frac{1}{\sqrt{nt}}\wedge 1)\exp(-\frac{c_2n(t-p)^2}{t})
\end{align}
Finally, there are only $O(1)$ good paths, and for $p\le\frac{1}{2}$ we arrive at
\begin{align}
& |\EE_p B_u| \nonumber \\
&\quad  \le \sum_{L}|B_u(L)| \cdot \mathbb{P}(L) \nonumber \\
&\quad \lesssim_u n^{-u}p^{u\wedge 1}\cdot (\frac{1}{\sqrt{nt}}\wedge 1)\exp(-\frac{c_2n(t-p)^2}{t}).
\end{align}

The previous approach can also be applied to the case where $\frac{1}{2}<p\le 1-\frac{1}{n}$, and we conclude that if $p\le \frac{1}{2}$,
\begin{align}\label{eq.remainder}
    |\EE_p B_u| \lesssim_u 
n^{-u}p^{u\wedge 1}\cdot (\frac{1}{\sqrt{nt}}\wedge 1)\exp(-\frac{c_2n(t-p)^2}{t});
\end{align}
if $\frac{1}{2}<p\le 1-\frac{1}{n}$, 
\begin{align}
    |\EE_p B_u| \lesssim_u & n^{-u}(1-p)^{u\wedge 1}\cdot (\frac{1}{\sqrt{n(1-t)}}\wedge 1)\nonumber \\ 
    & \cdot \exp(-\frac{c_2n(t-p)^2}{1-p})
\end{align}

Finally, when $t>1-\frac{1}{n}$, it's easy to see
\begin{align}
A_{n,u}(t) = (1-t)^u\cdot \mathbb{P}(\hat{p}_n=1) = (1-t)^up^n
\end{align}
and
\begin{align}\label{eq.large_t}
|\Delta^s A_{n,u}(t)| & = (1-t)^u|\Delta^s p^n|\nonumber \\ 
&\lesssim_{u,s}  (1-t)^up^{n-s}(1-p)^s.
\end{align}
Now the combination of \eqref{eq.non-remainder}, \eqref{eq.remainder} and \eqref{eq.large_t} completes the proof of Lemma \ref{lemma.A_nr}. \qed

\subsection{Proof of Lemma~\ref{lem.A_nr_converse}}

First note that by assumption we have
\begin{align}
\frac{k+1}{n-j} - t = \frac{1+O(p)}{n-j}
\end{align}
for any $j=0,\cdots,s$ in both cases, and when $u<s$ we have
\begin{align}
\frac{k}{n-s} - t \gtrsim \frac{p}{n}.
\end{align}
	
Adopt the same coupling as in the proof of Lemma \ref{lemma.A_nr}, and write
\begin{align}
& |\Delta^s A_{n,u}(t)| \nonumber \\
&\quad = |\EE_p \Delta^s (\hat{p}_n-t)_+^u(\mathbbm{1}(\hat{p}_{n-s}\in A) +\mathbbm{1}(\hat{p}_{n-s}\notin A)|\\
&\quad \ge |\EE_p \Delta^s (\hat{p}_n-t)_+^u\mathbbm{1}(\hat{p}_{n-s}\in A)| \nonumber \\ 
& \qquad -|\EE_p \Delta^s (\hat{p}_n-t)_+^u\mathbbm{1}(\hat{p}_{n-s}\notin A)|.
\end{align}
where
\begin{align}
A \triangleq \left[\frac{k-s}{n-s},\frac{k}{n-s}\right].
\end{align}
By the proof of Lemma \ref{lemma.A_nr}, the non-remainder terms in the Taylor expansion is at most $O(n^{-(\frac{u}{2}+s)})=o(n^{-(u+\frac{1}{2})})$, and by our coupling the remainder term is non-zero only if $\hat{p}_{n-s}\in A$, we conclude that
\begin{align}
|\EE_p \Delta^s (\hat{p}_n-t)_+^u\mathbbm{1}(\hat{p}_{n-s}\notin A)| = o(n^{-(u+\frac{1}{2})}).
\end{align}

Now we deal with the first term. It's easy to verify that for any $x\in A$, we have
\begin{align}
|x-p| = |t-p| + O(\frac{1}{n}) \le \frac{1}{\sqrt{n}}+O(\frac{1}{n}).
\end{align}
As a result, by the Stirling approximation formula in the proof of Lemma \ref{lemma.A_nr}, we conclude that for any $\frac{m}{n-s}\in A$ with $m\in\mathbb{N}$, 
\begin{align}
q&\triangleq\mathbb{P}(\hat{p}_{n-s} =\frac{k}{n-s}) \nonumber \\ 
& = (1+o(1))\cdot\mathbb{P}(\hat{p}_{n-s}=\frac{m}{n-s}) \gtrsim \frac{1}{\sqrt{n}}.
\end{align}
In other words, $\hat{p}_{n-s}$ is almost uniformly distributed restricted to $A$ with mass at least $\Theta(\frac{1}{\sqrt{n}})$ on each point. Moreover, for $j=0,1,\cdots,s$ and $i=1,\cdots,s-j$, it follows from the coupling that
\begin{align}
&\mathbb{P}(\hat{p}_{n-j}=\frac{k+i}{n-j},\hat{p}_{n-s}\in A) \nonumber\\
&=\sum_{m=0}^s \mathbb{P}(\hat{p}_{n-j}=\frac{k+i}{n-j}|\hat{p}_{n-s}=\frac{k-m}{n-s})\nonumber \\ 
& \quad \cdot\mathbb{P}(\hat{p}_{n-s}=\frac{k-m}{n-s})\\
&= \sum_{m=0}^s \binom{s-j}{m+i}p^{m+i}(1-p)^{s-j-m-i} \cdot q(1+o(1))
\end{align}

As a result, when $u<s$, we have
\begin{align}
&|\EE_p \Delta^s (\hat{p}_n-t)_+^u\mathbbm{1}(\hat{p}_{n-s}\in A)|\nonumber\\
= & \left|\EE_p \sum_{j=0}^s(-1)^j\binom{s}{j}(\hat{p}_{n-j}-t)_+^u\mathbbm{1}(\hat{p}_{n-s}\in A)\right|\\
= &\Bigg|(-1)^s(\frac{p}{n})^u+ \sum_{i=1}^{s-j} \sum_{j=0}^s(-1)^j\binom{s}{j}(\frac{k+i}{n-j}-t)_+^u\nonumber \\ & \cdot \mathbb{P}(\hat{p}_{n-j}=\frac{k+i}{n-j},\hat{p}_{n-s}\in A)\Bigg|\\
= &q(1+o(1))\Big|(-1)^s(\frac{p}{n})^u  + \sum_{i=1}^{s-j} \sum_{j=0}^s\sum_{m=0}^{s-i-j}(-1)^j\nonumber \\ 
& \cdot \binom{s}{j}\binom{s-j}{m+i} p^{m+i}(1-p)^{s-j-m-i}(\frac{k+i}{n-j}-t)^u\Big|\\
=& q(1+o(1))\left|(-1)^s(\frac{p}{n})^u+\sum_{i=1}^{s}\sum_{m=0}^{s-i} \binom{s}{m+i}p^{m+i}\right.\nonumber\\
&\cdot\sum_{j=0}^{s-i-m}(-1)^j\binom{s-m-i}{j}(1-p)^{s-m-i-j}\nonumber \\ 
& \left.\cdot (\frac{k+i}{n-j}-t)^u\right|\\
= &q(1+o(1))\left|(-1)^s(\frac{p}{n})^u+\sum_{i=1}^{s}\sum_{m=0}^{s-i} \binom{s}{m+i}\right.\nonumber\\
& \cdot p^{m+i}(i+O(p))^u\sum_{j=0}^{s-i-m}(-1)^j\binom{s-m-i}{j}\nonumber \\ 
& \left. \cdot (1-p)^{s-m-i-j}(\frac{1}{n-j})^u\right|\\
= & q(1+o(1))\left|(-1)^s(\frac{p}{n})^u+\sum_{i=1}^{s}\sum_{m=0}^{s-i} \binom{s}{m+i} \right. \nonumber \\ 
& \left. \cdot p^{m+i}(i+O(p))^u \Delta^{s-i-m}\frac{(1-p)^{s-m-i-x}}{(n-x)^u} \right|.
\end{align}
By the product rule of differentiation we know that
\begin{align}
& \Delta^{s-i-m}\frac{(1-p)^{s-m-i-x}}{(n-x)^u} \nonumber \\ 
=&  \frac{(-p)^{s-i-m}(1+O(p))}{n^{u}}\cdot (1+o(1)).
\end{align}
Plugging into the previous expression, we arrive at
\begin{align}
&|\EE_p \Delta^s (\hat{p}_n-t)_+^u\mathbbm{1}(\hat{p}_{n-s}\in A)|\nonumber\\
= &\frac{q(1+o(1))(1+O(p))}{n^u}\Big|(-1)^s\Theta(p^u) \nonumber \\ 
& +\sum_{i=1}^{s}\sum_{m=0}^{s-i} \binom{s}{m+i}p^{m+i}i^u (-p)^{s-i-m} \Big|\\
= &\frac{q(1+o(1))(1+O(p))}{n^u}\Big|(-1)^s\Theta(p^u) \nonumber \\ 
& +p^s\sum_{i=1}^{s}\sum_{m=0}^{s-i} \binom{s}{m+i}(-1)^{m+i}i^u \Big|\\
\gtrsim & \frac{q}{n^u}\cdot p^u(1+O(p))\\
\asymp & n^{-(u+\frac{1}{2})}\cdot p^u(1+O(p))
\end{align}
where we have used $u<s$. Hence, in this case by choosing $p_0$ small enough we arrive at the desired result.

When $u\ge s$, similarly we have
\begin{align}
&|\EE_p \Delta^s (\hat{p}_n-t)_+^u\mathbbm{1}(\hat{p}_{n-s}\in A)|\nonumber\\
=& \frac{q(1+o(1))(1+O(p))}{n^u}\Big|p^s\sum_{i=1}^{s}\sum_{m=0}^{s-i} \binom{s}{m+i}\nonumber \\ & \cdot (-1)^{m+i}i^u \Big|\\
\asymp & \frac{p^s(1+O(p))}{n^{u+\frac{1}{2}}}\cdot \left|\sum_{r=0}^s (-1)^r\binom{s}{r}\sum_{i=1}^r i^u\right|\\
= & \frac{p^s(1+O(p))}{n^{u+\frac{1}{2}}}\cdot \left|\Delta^s\sum_{i=1}^r i^u\right|.
\end{align}
Since $\sum_{i=1}^r i^u$ is a polynomial of $r$ with degree $u+1>s$, its $s$-backward difference is not zero, and thus the desired result also follows by choosing $p_0$ small enough.

\subsection{Proof of Lemma~\ref{lemma.jackknifelowerboundspeicalfunctions}}

We first prove the lower bound for the delete-$1$ jackknife. It suffices to fix $p = \frac{c}{n}$, where $c>0$ is a positive constant, and prove
\begin{align}
\lim_{n\to \infty} n \left ( \mathbb{E}_p \hat{f}_2 - f(p) \right ) \neq 0
\end{align} 
when $f(p) = -p\ln p$, and
\begin{align}
\lim_{n\to \infty} n^\alpha \left ( \mathbb{E}_p \hat{f}_2 - f(p) \right ) \neq 0
\end{align}
for $f(p) = p^\alpha, 0<\alpha<1$. 

It follows from Lemma~\ref{lemma.2jackkniferepres} that it suffices to analyze the expectation of the following quantity:
\begin{align}
g_n(X) = & n f\left( \frac{X}{n} \right) - \frac{n-1}{n} \Bigg( (n-X) f\left( \frac{X}{n-1}\right) \nonumber \\ 
& + X f\left( \frac{X-1}{n-1} \right) \Bigg) - f\left(\frac{c}{n} \right). 
\end{align}
As $n\to \infty$, the measure $\mu_n = \mathsf{B}(n,p)$ converges weakly to $\mu = \mathsf{Poi}(c)$. 

We will show that $\lim_{n\to \infty} n \mathbb{E}_{\mu_n} g_n(X) \neq 0$ when $f(p) = -p\ln p$ and $\lim_{n\to \infty} n^\alpha \mathbb{E}_{\mu_n} g_n(X) \neq 0$ when $f(p) = p^\alpha$.

For $f(p) = -p\ln p$, 
\begin{align}
& \mathbb{E}_{\mu_n} n \cdot g_n(X) \nonumber \\
 = & \mathbb{E}_{\mu_n} n \cdot \Big( n \frac{X}{n} \ln \frac{n}{X} - \frac{n-1}{n} \frac{(n-X)X}{n-1} \ln \frac{n-1}{X} \nonumber \\ 
 & - \frac{n-1}{n}\frac{X(X-1)}{n-1}\ln \frac{n-1}{X-1} - \frac{c}{n} \ln \frac{n}{c} \Big) \\
 = &  \mathbb{E}_{\mu_n} n \cdot \Big( X \ln \frac{n}{X} - \frac{X(n-X)}{n}\ln \frac{n-1}{X} \nonumber \\ & - \frac{X(X-1)}{n}\ln \frac{n-1}{X-1} - \frac{c}{n}\ln \frac{n}{c} \Big) \\
= & \mathbb{E}_{\mu_n} n \cdot \Big( X\ln n - \frac{X(n-X)}{n}\ln(n-1) \nonumber \\ & - \frac{X(X-1)}{n}\ln(n-1) - \frac{c}{n}\ln \frac{n}{c} \Big) \\
& + \mathbb{E}_{\mu_n} n\cdot \Big( X \ln \frac{1}{X} - \frac{X(n-X)}{n}\ln \frac{1}{X} \nonumber \\ 
&  - \frac{X(X-1)}{n}\ln \frac{1}{X-1} \Big) \\
= & \left( (c-cn)\ln \left(1-\frac{1}{n}\right) + c\ln c \right)  \nonumber \\ 
& +\mathbb{E}_{\mu_n} \left( X^2 \ln \frac{1}{X} - X(X-1)\ln \frac{1}{X-1} \right) \\
 = &  A_n + B_n. 
\end{align}
Since $\lim_{n\to \infty} A_n  = c + c \ln c$, $\lim_{c\to 0^+} c^{-2}(c + c\ln c) = \infty$, and $c^{-2} \mathbb{E}_\mu \left( X^2 \ln \frac{1}{X} - X(X-1)\ln \frac{1}{X-1} \right)$ has a finite limit at $c = 0^+$, we know there exists some $c>0$ such that 
\begin{align}
\lim_{n\to \infty} \mathbb{E}_{\mu_n} n \cdot g_n(X) \neq 0. 
\end{align}

For $f(p) = p^\alpha, 0<\alpha<1$, we have
\begin{align}
& \mathbb{E}_{\mu_n} n^\alpha \cdot g_n(X) \nonumber \\
 = & n^\alpha \Big( n \left(\frac{X}{n} \right)^\alpha - \frac{n-1}{n} (n-X) \left( \frac{X}{n-1} \right)^\alpha \nonumber \\ 
&  - \frac{n-1}{n}X \left(\frac{X-1}{n-1} \right)^\alpha - \left(\frac{c}{n} \right)^\alpha \Big) \\
 = &\mathbb{E}_{\mu_n} \Big( n X^\alpha - \left( \frac{n-1}{n} \right)^\alpha \left( X^\alpha(n-X) + X(X-1)^\alpha\right) \nonumber \\ 
 & - c^\alpha \Big) \\
 = & \mathbb{E}_{\mu_n} \left( n X^\alpha \left( 1- \left(1-\frac{1}{n} \right)^\alpha \right) - c^\alpha \right) \nonumber\\ 
 & + \left( \frac{n-1}{n} \right)^{1-\alpha}  \mathbb{E}_{\mu_n} \left( X^{\alpha+1}-X(X-1)^\alpha\right).
\end{align}
Hence, 
\begin{align}
\lim_{n\to \infty} \mathbb{E}_{\mu_n} n^\alpha \cdot g_n(X) & = \mathbb{E}_\mu \left( \alpha X^\alpha + X^{\alpha+1} - X(X-1)^\alpha \right) \nonumber \\ & \quad - c^\alpha. 
\end{align}
Noting that $c^{-1}\mathbb{E}_\mu \left( \alpha X^\alpha + X^{\alpha+1} - X(X-1)^\alpha \right)$ has a finite limit at $c = 0$ but $c^{\alpha-1}$ does not, we know there exists some $c>0$ such that 
\begin{align}
\lim_{n\to \infty} \mathbb{E}_{\mu_n} n^\alpha \cdot g_n(X) \neq 0
\end{align}
when $f(p) = p^\alpha,0<\alpha<1$.

For a $2$-jackknife that satisfies Condition~\ref{condi.boundedcoeffcondition}, the lower bound is given by Theorem~\ref{thm.jackknife} and~(\ref{eqn.dtmoduluscomputationexample}). 

\subsection{Proof of Theorem~\ref{thm.bootstrapafew}}

The first three statements follow from~\cite{gonska1994approximation}. The last statement follows from~\cite{Totik1994}. Now we work to prove the fourth statement. 

It follows from~\cite[Inequality 9.3.5., 9.3.7.]{Ditzian--Totik1987} that for $f\in C[0,1]$, we have $\| \varphi^{2m} (B_k[f])^{(2m)} \| \leq C k^{m} \| f \| $, and for $f$ such that $f^{(2m)} \in C[0,1]$, we have $\| \varphi^{2m} (B_k[f])^{(2m)} \| \leq C \| \varphi^{2m} f^{(2m)} \|$. Here $\varphi(x) = \sqrt{x(1-x)}$. 

It follows from the definition of the $K$-functional $K_{2m,\varphi}(f, 1/n^{m})$ that 
\begin{align}
K_{2m,\varphi}(f, 1/n^{m}) \leq & \| f - \oplus^{m} B_k[f] \| \nonumber \\ 
& + n^{-m} \| \varphi^{2m} (\oplus^{m} B_k[f])^{(2m)} \|  
\end{align}

Noting that $\oplus^{m}B_k[f]$ is nothing but linear combinations of the powers of $B_k$ up to degree $m$, we can apply the estimates on $\| \varphi^{2m} (B_k[f])^{(2m)} \|$ to $\| \varphi^{2m} (\oplus^{m} B_k[f])^{(2m)} \|  $. For any $g \in \text{A.C.}_{\text{loc}}$, we have
\begin{align}
& \| \varphi^{2m} (\oplus^{m} B_k[f])^{(2m)} \| \nonumber \\
 \leq& \| \varphi^{2m} (\oplus^{m} B_k[f-g])^{(2m)} \| + \| \varphi^{2m} (\oplus^{m} B_k[g])^{(2m)} \| \\
 \leq& C k^{m} \| f -g\| + C  \| \varphi^{2m} g^{(2m)} \|. 
\end{align}

Hence, 
\begin{align}
& n^{-m} \| \varphi^{2m} (\oplus^{m} B_k[f])^{(2m)} \| \nonumber \\
 \leq &C \inf_g \left( \frac{k}{n}\right)^{m} \left( \| f -g\|  + k^{-m} \| \varphi^{2m} g^{(2m)} \| \right) \\
 = & C \left( \frac{k}{n} \right)^{m} K_{2m,\varphi}(f, k^{-m}). 
\end{align}

We have showed that
\begin{align}\label{eqn.generalconverse}
K_{2m,\varphi}(f, 1/n^m) \leq & \| f - \oplus^{m} B_k[f] \|  \nonumber \\ 
& + C \left( \frac{k}{n} \right)^{m} K_{2m,\varphi}(f, k^{-m}). 
\end{align}

Now we utilize the assumption that $\omega_\varphi^{2m}(f,2t) \leq D \omega_\varphi^{2m}(f,t)$. We have
\begin{align}
D^{-q} \omega_\varphi^{2m} (f, 1/\sqrt{k}) & \leq \omega_\varphi^{2m}(f, 2^{-q} k^{-1/2}) \\
& = \omega_\varphi^{2m}(f, (2^{2q} k)^{-1/2}) \\
& \leq M_1 K_{2m, \varphi}(f, 1/ (2^{2q(m)} k^{m} )) ,
\end{align}
where in the last step we have used the equivalence between $K$-functional and the Ditzian--Totik modulus. 

Setting $n = 2^{2q}k$ and applying (\ref{eqn.generalconverse}), we have
\begin{align}
D^{-q} \omega_\varphi^{2m} (f, 1/\sqrt{k}) \leq  & M_1 \| f - \oplus^{m} B_k[f] \|\nonumber \\ 
& + M_1 2^{-2qm} \omega_\varphi^{2m}(f, 1/\sqrt{k}). 
\end{align}

We now choose $q$ so that $D^{-q} > 2M_1 2^{-2qm}$, which is possible since $D^{-1} > 2^{-2m}$, and obtain 
\begin{align}
\| f - \oplus^{m} B_k[f] \|  \geq 2^{-2qm} \omega_\varphi^{2m}(f, 1/\sqrt{k}). 
\end{align}
The proof is complete. 

\subsection{Proof of Lemma~\ref{lemma.taylorbiascorrectionbasic}}

It follows from~\cite[Theorem 1]{hurt1976asymptotic} that 
\begin{align}
\left | \mathbb{E}_p f(\hat{p}_n) - f(p) - \sum_{j = 1}^{2k-1} \frac{f^{(j)}(p)}{j!} \mathbb{E}_p \left( \hat{p}_n - p \right)^j \right | & \lesssim_{k,L} \frac{1}{n^k}.
\end{align}
It follows from Lemma~\ref{lemma.binomialmoments} and~\ref{lemma.centralmomentcoeffbounds} that for any $j\geq 1$ integer there exist polynomials $h_{m,j}(p)$ with degree no more than $j$, and coefficients independent of $n$ such that
\begin{align}
\mathbb{E}_p (\hat{p}_n - p)^j & = \sum_{m = 1}^{\lfloor j/2 \rfloor} h_{m,j}(p) \frac{1}{n^{j-m}}. 
\end{align}
Define
\begin{align}
T_j[f](p) & = \sum_{i = 1}^{2k-1} \frac{f^{(i)}(p)}{i!} h_{i-j,i}(p)\nonumber \\ & \quad \cdot \mathbbm{1}\left( i - \lfloor i/2 \rfloor \leq j \leq i-1 \right).
\end{align}
It is clear that $T_j[f](p)$ depends on $f$ only through the derivatives of $f$ of order from $j+1$ to $2j$. Concretely, $T_j[f](p)$ is a linear combination of the derivatives of $f$ of order from $j+1$ to $2j$ where the combination coefficients are polynomials of $p$ with degree no more than $2j$.

\subsection{Proof of Theorem~\ref{thm.withersbiascorrectionperformance}}
\begin{proof}
We first show, through induction, that each $t_i, 0\leq i\leq k-1$ satisfies condition $D_{2(k-i)}$ with parameter $\lesssim_{k} L$. Indeed, it is true for $i = 0$. Assuming that it is true for $i = m, 1\leq m\leq k-2$, we now show that it is true for $i = m+1$. We have
\begin{align}
t_{m+1}(p) = - \sum_{j =1}^{m+1} T_j[t_{m+1-j}](p). 
\end{align} 
Since $T_j[f]$ involves the derivatives of $f$ up to order $2j$ (Lemma~\ref{lemma.taylorbiascorrectionbasic}), for each $j$-th term $T_j[t_{m+1-j}](p)$, $1\leq j\leq m+1$, it involves the derivatives of $f$ up to order
\begin{align}
2(m+1-j) + 2j & = 2(m+1),
\end{align}
which implies that $t_{m+1}$ satisfies the condition $D_{2k - 2(m+1)}$ with parameter $\lesssim_{k} L$.

Now we apply Lemma~\ref{lemma.taylorbiascorrectionbasic} to each term in the formula of $\hat{f}_k$. For the $i$-th term, $0\leq i\leq k-1$, we have
\begin{align}
&\mathbb{E}_p \left[ \frac{1}{n^i} t_i(\hat{p}_n) \right] \nonumber \\= & \frac{1}{n^i} t_i(p) + \sum_{j = 1}^{k-i-1} \frac{1}{n^{j+i}} T_j[t_i](p) + O \left( \frac{1}{n^k} \right) \\
 = &\frac{1}{n^i} t_i(p) + \sum_{m = i+1}^{k-1} \frac{1}{n^{m}} T_{m-i}[t_i](p) + O \left( \frac{1}{n^k} \right).
\end{align}
Sum over $0\leq i\leq k-1$, we have
\begin{align}
\mathbb{E}_p[\hat{f}_k] & = \sum_{i = 0}^{k-1} \left( \frac{1}{n^i} t_i(p) + \sum_{m = i+1}^{k-1} \frac{1}{n^{m}} T_{m-i}[t_i](p) \right) \nonumber \\& \quad + O \left( \frac{1}{n^k} \right).
\end{align}
It suffices to show that
\begin{align}
\sum_{i = 0}^{k-1} \left( \frac{1}{n^i} t_i(p) + \sum_{m = i+1}^{k-1} \frac{1}{n^{m}} T_{m-i}[t_i](p) \right) & = f(p). 
\end{align}
We have
\begin{align}
\sum_{i = 0}^{k-1} \left( \frac{1}{n^i} t_i(p) \right) & = f(p) + \sum_{i = 1}^{k-1} \frac{1}{n^i} t_i(p),
\end{align}
and 
\begin{align}
\sum_{i = 0}^{k-1} \sum_{m = i+1}^{k-1} \frac{1}{n^{m}} T_{m-i}[t_i](p) & = \sum_{m = 1}^{k-1} \frac{1}{n^m} \sum_{i = 0}^{m-1} T_{m-i}[t_i](p) \\
& =  -\sum_{m = 1}^{k-1} \frac{1}{n^m} t_m(p),
\end{align}
where in the last step we used the definition of $t_i(p)$. The proof is now complete. 
\end{proof}

\section{Proofs of auxiliary lemmas}\label{sec.proofofauxlemmas}

\subsection{Proof of Lemma~\ref{lemma.meanvaluetheorem}}

Let $p(x)$ be the Lagrangian interpolating polynomial of $f$ at points $x_0,\cdots,x_r$, and define $g(x)=f(x)-p(x)$. It is easy to see that $g(x)$ has at least $(r+1)$ zeros $x_0,\cdots,x_r$. Since
\begin{align}
0 = g(x_1) - g(x_0) = \sum_{x=x_0+1}^{x_1} \Delta g(x)
\end{align}
there must be some $y_0\in (x_0,x_1]$ such that $\Delta g(y_0)\le 0$. Similarly, there exists some $y_1\in (x_1,x_2]$ such that $\Delta g(y_1)\ge 0$, and $y_2\in (x_2,x_3]$ such that $\Delta g(y_2)\le 0$, so on and so forth. Next observe that
\begin{align}
0 \le \Delta g(y_1) - \Delta g(y_0) = \sum_{y=y_0+1}^{y_1} \Delta^2g(y)
\end{align}
there must exist some $z_0\in (y_0,y_1]$ such that $\Delta^2g(z_0)\ge 0$. Similarly, there exists some $z_1\in (y_1,y_2]$ such that $\Delta^2g(z_1)\le 0$, and $z_2\in (y_2,y_3]$ such that $\Delta^2g(z_2)\ge 0$, so on and so forth. Repeating this process, there must be some $x\in [x_0,x_r]$ such that $(-1)^r\Delta^rg(x)\ge 0$. Note that $p(x)$ is a degree-$r$ polynomial with leading coefficient $f[x_0,\cdots,x_r]$, we conclude that
\begin{align}
0&\le (-1)^r\Delta^rg(x)\\
& = (-1)^r\Delta^r(f(x)-p(x))\\
& = (-1)^r(\Delta^rf(x)-r!f[x_0,\cdots,x_r]).
\end{align}

Similarly, we can also show that there exists $x'\in [x_0,x_r]$ such that
\begin{align}
0\le (-1)^{r+1}(\Delta^rf(x')-r!f[x_0,\cdots,x_r]).
\end{align}
Combining these two inequalities completes the proof.

\subsection{Proof of Lemma~\ref{lemma.higherpower}}

The key idea is to connect the problem of solving the matrix equations~(\ref{eqn.vonderm}) to the notion of divided difference in approximation theory. 

For $\rho \geq 0, \rho \in \mathbb{Z}$, 
\begin{align}
\sum_{i = 1}^r \frac{C_i}{n_i^\rho} & = \sum_{i = 1}^r \frac{1}{n_i^\rho} \prod_{j\neq i} \frac{n_i}{n_i - n_j} \\
& = \sum_{i = 1}^r n_i^{r-1-\rho} \prod_{j\neq i} \frac{1}{n_i - n_j} \\
& = x^{r-1-\rho}[n_1,n_2,\ldots, n_r],
\end{align}
where $f[x_1,x_2,\ldots,x_r]$ denotes the divided difference in Definition~\ref{def.divideddifference}. The lemma is proved using the mean value theorem (Lemma~\ref{lemma.meanvaluetheoremdivideddifferrence}) of the divided difference for function $x^{r-1-\rho}$.

\subsection{Proof of Lemma~\ref{lemma.centralmomentcoeffbounds}}

The first part follows from Lemma~\ref{lemma.binomialmoments}. Regarding the second part, the moment generating function of $\hat{p}_n-p$ is given by
\begin{align}
\EE[\exp(z(\hat{p}_n-p))] = e^{-zp}\left(1+p(e^{z/n}-1)\right)^n.
\end{align}
Written as formal power series of $z$, the previous identity becomes
\begin{align}
& \sum_{s=0}^\infty \frac{\EE(\hat{p}_n-p)^s}{s!}z^s  \nonumber \\
& \quad = \left(\sum_{i=0}^\infty \frac{(-p)^i}{i!}z^i\right)\left[\sum_{k=0}^n \binom{n}{k}p^k\left(\sum_{l=1}^\infty \frac{1}{l!}(\frac{z}{n})^l\right)^{k}\right].
\end{align}

Hence, by comparing the coefficient of $n^{j-s}z^s$ at both sides, we obtain
\begin{align}
& \frac{h_{j,s}(p)}{s!} \nonumber \\
 = &  \sum_{i=0}^j \frac{(-p)^i}{i!} \left(\sum_{k=j-i}^{s-i} \frac{t_{k,j-i}}{k!}p^k\sum_{\substack{a_1+\cdots+a_k=s-i, \\ a_1,\cdots,a_k\ge 1}} \prod_{l=1}^k \frac{1}{a_l!}\right)
\end{align}
where $t_{k,r}$ is the coefficient of $x^r$ in $x(x-1)\cdots(x-k+1)$. It's easy to see
\begin{align}
|t_{k,r}| \le k^{k-r}\binom{k}{r} \le \frac{k^k}{r!}. 
\end{align}
Moreover, it's easy to see when $k\le s-i$, we have
\begin{align}
& \sum_{a_1+\cdots+a_k=s-i, a_1,\cdots,a_k\ge 1} \prod_{l=1}^k \frac{1}{a_l!} \nonumber \\
&\quad \le \sum_{a_1+\cdots+a_k=s-i} \prod_{l=1}^k \frac{1}{a_l!} \\
&\quad \le \sum_{a_1+\cdots+a_{k+i}=s} \prod_{l=1}^{k+i} \frac{1}{a_l!} = \frac{(k+i)^s}{s!}
\end{align}
and this quantity is zero when $k>s-i$. 

Then, applying $k! \geq \left( \frac{k}{e} \right)^k$ yields
\begin{align}
|h_{j,s}(p)| & \leq \sum_{i = 0}^j \frac{1}{i!} \sum_{k = j-i}^{s-i} \frac{e^k (k+i)^s}{(j-i)!} \\
& \leq \sum_{i = 0}^j \frac{s e^s s^s}{i!(j-i)!} \\
& = s(es)^s \frac{2^j}{j!} \\
& \leq \frac{s(2es)^s}{j!} \\
& \leq \frac{(4es)^s}{j!}
\end{align}

\subsection{Proof of Lemma~\ref{lemma.strukov}}

We first prove the statement when the domain of $f$ is the whole real line. We introduce the first and second Stekolv functions $f_h(x), f_{hh}(x)$ as follows:
\begin{align}
f_h(x) & = f \ast K_h \\
f_h(x) & = f_h \ast K_h = f \ast K_h \ast K_h,
\end{align}
where $K_h = \frac{1}{h} \mathbbm{1}(x\in [-h/2,h/2])$ is the box kernel, and the operation $\ast$ denotes convolution. 

The Steklov functions have the following nice properties~\cite[Chap. V, Sec. 83]{akhiezer1956theory}:
\begin{align}
f(x) - f_h(x) & = \frac{1}{h} \int_{-h/2}^{h/2} (f(x) - f(x+t)) dt \\
f'_h(x) & = \frac{1}{h} \left( f\left( x+ \frac{h}{2} \right) - f\left( x- \frac{h}{2} \right) \right) \\
f(x) - f_{hh}(x) & = \frac{1}{h^2} \int_0^{h/2} \int_0^{h/2} [ 4f(x) - f(x+s+t) \nonumber \\ 
& \quad - f(x+s-t) - f(x-s+t) \nonumber \\ 
& \quad - f(x-s-t)] ds dt \\
f''_{hh}(x) & = \frac{1}{h^2} \left( f(x+h) + f(x-h) -2f(x) \right). 
\end{align}

Hence, we have
\begin{align}
& |f(x) - f_{hh}(x)| \nonumber \\ 
\leq &   \frac{1}{h^2} \int_0^{h/2} \int_0^{h/2} [ | 2f(x)- f(x+s+t)   - f(x-s-t)|  \nonumber \\ 
& +|2f(x) - f(x+s-t) - f(x-s+t)|] ds dt \\ 
 \leq & \frac{1}{h^2} \int_0^{h/2} \int_0^{h/2} [ 2 \omega^2(f,h)] dsdt \\
= & \frac{1}{2} \omega^2(f,h),
\end{align}
and 
\begin{align}
\| f''_{hh}(x) \| & \leq \frac{1}{h^2} \omega^2(f,h). 
\end{align}

We use
\begin{align}
&| \mathbb{E}[f(X)] - f(\mathbb{E}[X])| \nonumber \\
 \leq &| \mathbb{E}[f(X) - f_{hh}(X) + f_{hh}(X) - f_{hh}(\mathbb{E}[X]) \nonumber \\ &\quad + f_{hh}(\mathbb{E}[X]) - f(\mathbb{E}[X])] | \\
\leq & 2 \| f - f_{hh}\| + | \mathbb{E}[f_{hh}(X)] - f_{hh}(\mathbb{E}[X]) | \\
\leq & 2 \| f- f_{hh} \| + \frac{1}{2} \| f''_{hh}\| \mathsf{Var}(X) \\
 \leq& \omega^2(f,h) + \frac{1}{2h^2} \omega^2(f,h) \mathsf{Var}(X) \\
  = & 3 \cdot \omega^2(f, h),
\end{align}
where $h = \frac{\sqrt{\mathsf{Var}(X)}}{2}$. 

Now we argue that when the domain of $f$ is an interval $[a,b]$ that is a strict subset of $\mathbb{R}$, we can replace the constant $3$ by $15$. Indeed, as argued in~\cite[Sec. 3.5.71, pg. 121]{timan2014theory}, for any continuous function $f\in [a,b]$, one can extend $\phi(x) = f(x) - \frac{f(b)-f(a)}{b-a}(x-a) - f(a)$ to the whole real line while ensuring the second order modulus of the extension is upper bounded by five times the $\omega^2(f,t)$ of the original function $f$. Indeed, one achieves this by extending $\phi$ so as to be odd with respect to the ends of the $[a,b]$ and then periodically with period $2(b-a)$ on the whole real line.

\bibliographystyle{IEEEtran}
\bibliography{di}

\begin{IEEEbiographynophoto}{Jiantao Jiao}
(S'13-M'18) is an Assistant Professor in the Department of Electrical Engineering and Computer Sciences at the University of California, Berkeley. He received his Ph.D. and M.S. degrees from Stanford University in 2014 and 2018, respectively, and he received the B.Eng degree with the highest honor in Electronic Engineering from Tsinghua University, Beijing, China, in 2012. His research interests are in statistical machine learning, mathematical data science, optimization, applied probability, information theory, and their applications in science and engineering.  
\end{IEEEbiographynophoto}

\begin{IEEEbiographynophoto}{Yanjun Han}
(S'14) received his B.Eng. degree with the highest honor in
Electronic Engineering from Tsinghua University, Beijing, China in 2015, and a Master's degree in Electrical Engineering from Stanford University in 2017. He is currently working towards the Ph.D. degree in the Department of
Electrical Engineering at Stanford University. His research interests include information theory and statistics, with applications in communications, data compression, and learning.
\end{IEEEbiographynophoto}

\end{document}